\documentclass[11pt]{amsart}
\usepackage[utf8]{inputenc}

\usepackage{mystyle}

\usepackage[foot]{amsaddr}

\def\vincent#1{}
\def\dongyeap#1{}
\def\tom#1{}

\begin{document}

\title[Perfect matchings in random sparsifications of Dirac hypergraphs]
{
    Perfect matchings in random sparsifications of Dirac hypergraphs 
}

\author[Kang]{Dong Yeap Kang}
\address{Dong Yeap Kang, Extremal Combinatorics and Probability Group (ECOPRO), Institute for Basic Science (IBS), Daejeon, South Korea}
\email{\texttt{dykang.math@ibs.re.kr}}

\author[Kelly]{Tom Kelly}
\address{Tom Kelly, School of Mathematics, Georgia Institute of Technology, Atlanta, GA 30332, USA}
\email{\texttt{tom.kelly@gatech.edu}}

\author[K\"uhn]{Daniela K\"uhn}
\address{Daniela K\"uhn, School of Mathematics, University of Birmingham,
Edgbaston, Birmingham, B15 2TT, United Kingdom}
\email{\texttt{D.Kuhn@bham.ac.uk}}

\author[Osthus]{Deryk Osthus}
\address{Deryk Osthus, School of Mathematics, University of Birmingham,
Edgbaston, Birmingham, B15 2TT, United Kingdom}
\email{\texttt{D.Osthus@bham.ac.uk}}

\author[Pfenninger]{Vincent Pfenninger}
\address{Vincent Pfenninger, Institute of Discrete Mathematics, Graz University of Technology, Graz, Austria}
\email{\texttt{pfenninger@math.tugraz.at}}


\thanks{This project has received partial funding from the European Research
Council (ERC) under the European Union's Horizon 2020 research and innovation programme (grant agreement no.\ 786198, D.~Kang, T.~Kelly, D.~K\"uhn, D.~Osthus, and V.~Pfenninger).
Dong Yeap Kang was supported by Institute for Basic Science (IBS-R029-Y6).}


\begin{abstract}
For all integers $n \geq k > d \geq 1$, let $m_{d}(k,n)$ be the minimum integer $D \geq 0$ such that every $k$-uniform $n$-vertex hypergraph $\mathcal H$ with minimum $d$-degree $\delta_{d}(\mathcal H)$ at least $D$ has an optimal matching.
For every fixed integer $k \geq 3$, we show that for $n \in k \mathbb{N}$ and $p = \Omega(n^{-k+1} \log n)$, if $\mathcal H$ is an $n$-vertex $k$-uniform hypergraph with $\delta_{k-1}(\mathcal H) \geq m_{k-1}(k,n)$, then a.a.s.\ its $p$-random subhypergraph $\mathcal H_p$ contains a perfect matching. 
Moreover, for every fixed integer $d < k$ and $\gamma > 0$, we show that the same conclusion holds if $\mathcal H$ is an $n$-vertex $k$-uniform hypergraph with $\delta_d(\mathcal H) \geq m_{d}(k,n) + \gamma\binom{n - d}{k - d}$.
Both of these results strengthen Johansson, Kahn, and Vu's seminal solution to Shamir's problem and can be viewed as ``robust'' versions of hypergraph Dirac-type results.
In addition, we also show that in both cases above, $\mathcal H$ has at least $\exp((1-1/k)n \log n - \Theta (n))$ many perfect matchings, which is best possible up to an $\exp(\Theta(n))$ factor.
\end{abstract}

\maketitle



\section{Introduction}

A \emph{hypergraph} is an ordered pair $\cH = (V,E)$ of a set $V \coloneqq V(\cH)$ of \emph{vertices} of $\cH$ and a set $E \coloneqq E(\cH)$ of subsets of $V$, where the elements of $E$ are called the \emph{edges} of $\cH$. If $E(\cH) \subseteq \binom{V}{k}$ for some positive integer $k$, then we call $\cH$ \emph{$k$-uniform}. 
We often identify~$E(\cH)$ with~$\cH$ if its set of vertices is clear.
A \emph{matching} of a hypergraph $\cH$ is a set of disjoint edges of $\cH$. An \emph{optimal matching} of a $k$-uniform hypergraph $\cH$ is a matching consisting of $\lfloor |V(\cH)|/k \rfloor$ edges. 
An optimal matching of a $k$-uniform hypergraph $\cH$ is called \emph{perfect} if $k$ divides $|V(\cH)|$.

In a seminal paper by Edmonds~\cite{edmonds1965}, it is proved that there exists a polynomial-time algorithm to determine whether a given graph has a perfect matching. However, for $k \geq 3$, it is NP-complete to decide whether a given $k$-uniform hypergraph has a perfect matching (see~\cite{GareyJohnson, K72}). Thus, it is natural to consider sufficient conditions which force a perfect matching; a minimum degree condition, which is called a \emph{Dirac-type} condition because of Dirac's~\cite{dirac1952} classical result on Hamilton cycles in graphs, is one of the most intensively studied~\cite{rr2010, zhao2016}.  
Perfect matchings in \textit{random} graphs and hypergraphs have also attracted considerable interest.  The so-called \textit{Shamir's problem} (see~\cite{erdos1981}) of determining the threshold for the existence of a perfect matching in a random $k$-uniform hypergraph was considered one of the most important problems in probabilistic combinatorics before its resolution by Johansson, Kahn, and Vu~\cite{JKV2008} in 2008.  
Our results in this paper connect these two streams of research.

\subsection{Perfect matchings in Dirac hypergraphs}\label{subsec:pm}

For $d\in \mathbb N$, the \emph{minimum $d$-degree} $\delta_{d}(\cH)$ of a hypergraph~$\cH$ is the minimum of $|\{e \in \cH : v_1,\dots,v_{d} \in e \}|$ among all choices of $d$ distinct vertices $v_1 , \dots , v_{d} \in V(\cH)$. If $\cH$ is $k$-uniform, we also call $\delta_{k-1}(\cH)$ the \emph{minimum codegree}.
For $n \geq k > d \geq 1$, let $m_d(k,n)$ be the minimum integer $D \geq 0$ such that every $k$-uniform $n$-vertex hypergraph $\cH$ with $\delta_{d}(\cH) \geq D$ has an optimal matching, and for each $s \in \{0, \dots, k - 1\}$, let
\begin{equation*}
    \overline{\mu_d}^{(s)} (k) \coloneqq \limsup_{\substack{n \to \infty\\n \equiv s \:{\rm mod}\:k}} \frac{m_d(k,n)}{\binom{n-d}{k-d}}.
\end{equation*}

Determining the value of $m_d(k, n)$, or even just $\overline{\mu_d}^{(s)}(k)$ in many cases, is a major open problem.
R\"{o}dl, Ruci\'nski, and Szemer\'edi~\cite{RRS2008} first proved that $m_{k-1}(k,n) \leq n/2 + o(n)$ for $n \in k\mathbb N$ (in fact, they showed a tight Hamilton cycle exists if this codegree condition holds).
This bound was improved by K\"uhn and Osthus~\cite{KO2006} to $n/2 + 3k^2 \sqrt{n \log n}$, and R\"{o}dl, Ruci\'nski, and Szemer\'edi~\cite{RRS2006_2} improved it further to $n/2 + O(\log n)$.
Finally, R\"{o}dl, Ruci\'nski, and Szemer\'edi~\cite{Rodl2009} determined $m_{k-1}(k,n) = n/2 - k + C(k,n)$ for all sufficiently large $n \in k \mathbb N$, with $C(k,n) \in \{ 3/2 , 2 , 5/2 , 3 \}$ depending on $k$ and $n$. 
R\"{o}dl, Ruci\'nski, Schacht, and Szemer\'edi~\cite{rrms2008} also gave a simple proof for a bound of $n/2 + k/4$, that does not require $n$ to be large.

For $1 \leq d \leq 3k/8$ and $n \in k \mathbb{N}$, both the exact and the asymptotic values of $m_d(k,n)$ are unknown for many cases.
The exact value of $m_d(k,n)$ is known for $d \geq 3k/8$ and large $n \in k\mathbb{N}$ by a combination of results~\cite{FK18, TZ2016}, where  $m_d(k,n) = (1/2 + o(1)) \binom{n-d}{k-d}$ and the exact bound of $m_d(k,n)$ follows from the obstructions called~\emph{divisibility barriers}.
Khan~\cite{khan2013} and independently K\"uhn, Osthus, and Treglown~\cite{KOT2013} showed that $m_1(3,n) = \binom{n-1}{2} - \binom{2n/3}{2}$ for large $n \in 3 \mathbb{N}$. Khan~\cite{khan2016} showed that $m_1(4,n) = \binom{n-1}{3} - \binom{3n/4}{3}$ for large $n \in 4 \mathbb{N}$.
Alon, Frankl, Huang, R\"{o}dl, Ruci\'{n}ski, and Sudakov~\cite{AFHRRS2012} related the asymptotics of $m_d(k,n)$ and $m_d^*(k,n)$, where $m_d^*(k,n)$ is the minimum $D$ such that every $n$-vertex $k$-uniform hypergraph with minimum $d$-degree at least $D$ has fractional matching number $n/k$.
Ferber and Kwan~\cite{FK2022} showed that the limit of $m_d(k,n) / \binom{n-d}{k-d}$ exists as $n \in k \mathbb{N}$ tends to infinity, and it is conjectured~\cite{HPS2009, KO2009} that $\overline{\mu_d}^{(0)}(k) = \max \{ 1/2 , 1 - (\frac{k-1}{k})^{k-d} \}$.  See \cite[Conjecture 1.5]{zhao2016} for the exact conjectured value of $m_d(k, n)$ when $n \in k \mathbb N$ is large.

For the other case $k \nmid n$, R\"{o}dl, Ruci\'nski, and Szemer\'edi~\cite{Rodl2009} showed that $m_{k-1}(k,n) \leq n/k + O(\log n)$, and Han~\cite{Han2015} determined $m_{k-1}(k,n) = \lfloor n/k \rfloor$ for all sufficiently large $n$ (and thus $\overline{\mu_{k-1}}^{(s)}(k) = 1 / k$ for all $s \neq 0$).  
Han~\cite[Conjecture 1.10]{Han2016} conjectured an upper bound on $\overline{\mu_d}^{(s)}(k)$ for all $d,s \in [k - 1]$ and proved a matching lower bound on $m_d(k, n)$.  (Thus, if true, Han's conjecture implies the limit of $m_d(k,n) / \binom{n-d}{k-d}$ exists as $n \in k \mathbb{N} + s$ tends to infinity.)
See~\cite{CGHW22, Han2016, LYY21} for more background on the non-divisible case, and for more discussion of this topic, see  the surveys~\cite{rr2010, zhao2016}.

\subsection{Perfect matchings in random hypergraphs}

A \emph{random $k$-uniform $n$-vertex hypergraph} $\cH^k(n,p(n))$ is a $k$-uniform hypergraph on $n$ vertices obtained by choosing each subset of $k$ vertices to be an edge with probability $p(n)$ independently at random.
Regarding the existence of a perfect matching in a random hypergraph, it is natural to ask for the threshold for $\cH^k(kn,p(n))$ to contain a perfect matching.

For $k=2$, in a seminal paper Erd\H{o}s and R\'{e}nyi~\cite{ER1966} determined the (sharp) threshold for $\cH^2(2n,p(n))$ to contain a perfect matching. They showed that the probability that $\cH^k(2n,p(n))$ has a perfect matching tends to 1 if $p(n) = \frac{\log n + \omega(1)}{2n}$ and tends to 0 if $p(n) = \frac{\log n - \omega(1)}{2n}$.

On the other hand, for $k \geq 3$, it is much more difficult to determine the threshold for the appearance of a perfect matching.
In 1979, Shamir (see~\cite{erdos1981, SS1983})\COMMENT{Erd\H{o}s~\cite{erdos1981} only posed it in the $3$-uniform case, but according to Schmidt and Shamir~\cite{SS1983}, they asked him the more general $k$-uniform case.} asked for the threshold for $\cH^k(kn,p(n))$ to contain a perfect matching (a precise and explicit statement was mentioned in~\cite{CFMR1996}). 
Schmidt and Shamir~\cite{SS1983} showed that \textit{asymptotically almost surely} (which we abbreviate as a.a.s.)  $\cH^k(kn , p(n))$ has a perfect matching if $p(n) = \omega(n^{-k + 3/2})$. This was further improved by Frieze and Janson~\cite{FJ1995} to $p(n) = \omega(n^{-k + 4/3})$.
Finally, Johansson, Kahn, and Vu~\cite{JKV2008} proved that the threshold for $\cH^k(kn,p(n))$ to contain a perfect matching is $\Theta(n^{-k+1} \log n)$, matching the threshold for $\cH^k(kn,p(n))$ not to contain an isolated vertex. Recently, Kahn~\cite{kahn2019} determined the sharp threshold for $\cH^k(kn,p(n))$ to contain a perfect matching, as well as the hitting time result~\cite{kahn2022}, which proves the conjecture in~\cite{CFMR1996} in a stronger form.

\subsection{Robust version of Dirac-type theorems}
For any hypergraph $\cH$ and $p \in [0,1]$, let $\cH_p$ be a  spanning random subhypergraph of $\cH$ obtained by choosing each edge $e \in \cH$ with probability $p$ independently at random.
The problem of determining whether a certain property of the original hypergraph~$\cH$ is retained by~$\cH_p$ has been studied extensively~\cite{ABCDJMRS2022, AK2022, CEGKO2020, GLS2021, johansson2020, KLS2014, KLS2015, PSSS22}, and results of this nature are referred to as \textit{robustness} results~\cite{sudakov2017}.
For example, 
Krivelevich, Lee, and Sudakov~\cite{KLS2014} showed a robust version of Dirac's theorem that for every $n$-vertex graph $G$ with minimum degree at least $n/2$, \aas its random subgraph $G_p$ contains a Hamilton cycle for $p = p(n) \geq C \log n / n$ for some absolute constant $C>0$,
providing a common generalization of Dirac's theorem~\cite{dirac1952} (when $p = 1$) and the classic result of P\'{o}sa~\cite{posa1976} (when $G \cong K_n)$ on the threshold for the appearance of a Hamilton cycle in a random graph.



Our first result is the following robust version of hypergraph Dirac-type results on $m_d(k, n)$ in the general case $1 \leq d \leq k-1$.  Here, the integer $n$ is not necessarily divisible by $k$.
The case $k=2$ follows from the result by Krivelevich, Lee, and Sudakov~\cite{KLS2014} on the robust Hamiltonicity.
 

\begin{theorem}\label{thm:main_general}
Let $d,k,s \in \mathbb{Z}$ such that $k \geq 3$, $1 \leq d \leq k-1$, and $0 \leq s \leq k - 1$.
For every $\gamma > 0$, there exists $C>0$ such that the following holds for $n \in k\mathbb{N} + s$  and $p = p(n) \in [0,1]$ with $p \geq C \log n / n^{k-1}$.
If $\cH$ is a $k$-uniform $n$-vertex hypergraph with $\delta_{d}(\cH) \geq \left(\overline{\mu_d}^{(s)}(k) + \gamma\right) \binom{n-d}{k-d}$, then \aas a random subhypergraph~$\cH_p$ contains an optimal matching.
\end{theorem}

Combining this result with the aforementioned prior work determining $m_d(k, n)$ \cite{AFHRRS2012, FK18, Rodl2009, TZ2016}, we simultaneously obtain that for every $\gamma > 0$, as $n\rightarrow\infty$, for $p = \Omega(\log n / n^{k - 1})$,  $\cH_p$ \aas has a perfect matching when $\cH$ is a $k$-uniform $n$-vertex hypergraph satisfying $\delta_d(\cH) \geq (1/2 + \gamma)\binom{n - d}{k - d}$ for some $d \geq 3k / 8$ when $k \mid n$ and that $\cH_p$ \aas has an optimal matching when $\cH$ is a $k$-uniform $n$-vertex hypergraph satisfying $\delta_{k-1}(\cH) \geq (1/k + \gamma)n$ when $k \nmid n$.  
Another interesting feature of this result is that it implies the existence of optimal matchings in random sparsifications of hypergraphs with minimum $d$-degree at least $\left(\overline{\mu_d}^{(s)}(k) + \gamma\right) \binom{n-d}{k-d}$ even in the cases in which the value of $\overline{\mu_d}^{(s)}(k)$ is not known.  
Since $\lim_{n \rightarrow\infty : k \mid n}\left.m_d(k,n)\middle / \binom{n-d}{k-d}\right. = \overline{\mu_d}^{(0)} (k)$~\cite{FK2022}, for $n \in k \mathbb{N}$, the minimum degree condition in Theorem~\ref{thm:main_general} can be replaced by $\delta_d(\cH) \geq m_d(k, n) + \gamma \binom{n - d}{k - d}$.

Our main result is the following robust version of the Dirac-type result by R\"{o}dl, Ruci\'nski, and Szemer\'edi~\cite{Rodl2009}.
\begin{theorem}\label{thm:main_rrs}
Let $k \geq 3$ be an integer. There exists $C>0$ such that the following holds for $n \in k\mathbb{N}$ and $p=p(n) \in [0,1]$ with $p \geq C \log n / n^{k-1}$.
If $\cH$ is a $k$-uniform $n$-vertex hypergraph with $\delta_{k-1}(\cH) \geq m_{k-1}(k,n)$, then \aas a random subhypergraph $\cH_p$ contains a perfect matching.
\end{theorem}

The value of $p$ in both Theorems~\ref{thm:main_general} and~\ref{thm:main_rrs} is asymptotically best possible, since it is well known that \aas there are $\omega(1)$ isolated vertices in a random $k$-uniform $n$-vertex hypergraph $\cH^k(n,p)$ if $p \leq \frac{(k-1)! \log n - \omega(1)}{n^{k-1}}$.  In fact, both results generalize Johansson, Kahn, and Vu's~\cite{JKV2008} solution to Shamir's problem that the threshold for the existence of a perfect matching in $\cH^k(kn, p(n))$ is $\Theta(n^{-k+1}\log n)$.  

We remark that Theorems~\ref{thm:main_general} and~\ref{thm:main_rrs} are implied by Theorems~\ref{thm:main_general_fknp} and~\ref{thm:main_rrs_fknp} below, respectively.

\subsection{Spreadness and a lower bound on the number of perfect matchings}

To prove Theorems~\ref{thm:main_general} and \ref{thm:main_rrs}, we use the fractional version of the Kahn--Kalai conjecture~\cite{kk2006} (conjectured by Talagrand~\cite{Ta10}), recently  resolved by Frankston, Kahn, Narayanan, and Park~\cite{Frankston2021}. The Kahn--Kalai conjecture was recently proved in full by Park and Pham~\cite{pp2022}, but the fractional version is sufficient for our application.  A precursor to these results was the main technical ingredient in Alweiss, Lovett, Wu, and Zhang's~\cite{ALWZ2021} breakthrough on the Erd\H{o}s--Rado sunflower conjecture~\cite{ER60}, and the results have been found to have many additional applications.  
This paper, and the independent work of Pham, Sah, Sawhney, and Simkin~\cite{PSSS22} (discussed further in the remark at the end of this subsection), are the first to demonstrate an application of the result to robustness of Dirac-type results. 

The Frankston--Kahn--Narayanan--Park theorem implies the Johansson--Kahn--Vu solution to Shamir's problem.  Moreover, it reduces our problem to proving that there exists a probability measure on the set of perfect or optimal matchings that is `well-spread'.  Roughly speaking, this means that the probability measure chooses a perfect matching at random in such a way that no particular set of edges is very likely to be contained in the matching.




\begin{definition}[Spreadness]
    Let $\cH$ be a $k$-uniform hypergraph and $q \in [0,1]$. Let $\nu$ be a probability measure on the set of matchings of $\cH$, and let $M$ be a matching in $\cH$ chosen at random according to $\nu$. We say that $\nu$ is \emph{$q$-spread} if for each $s \geq 1$ and $e_1, \dots, e_s \in \cH$, we have 
    \[
        \pr{e_1, \dots, e_s \in M} \leq q^s.
    \]
\end{definition}


The next theorem follows from~\cite[Theorem 1.6]{Frankston2021}. More precisely, it follows from the derivation of~\cite[Theorem 1.1]{Frankston2021} from~\cite[Theorem 1.6]{Frankston2021}.

\begin{theorem}[Frankston, Kahn, Narayanan, and Park~\cite{Frankston2021}]\label{thm:fknp}
    There exists $K > 0$ such that the following holds.
    Let $\cH$ be a $k$-uniform $n$-vertex hypergraph and $q \in [0,1]$. 
    If there exists a $q$-spread probability measure on the set of optimal matchings of $\cH$ and $p \geq K q \log n$, then \aas there exists an optimal matching in $\cH_p$.
\end{theorem}

In particular, by Theorem~\ref{thm:fknp}, it suffices to prove the following results to deduce Theorems~\ref{thm:main_general} and~\ref{thm:main_rrs}, respectively.

\begin{theorem}\label{thm:main_general_fknp}
Let $d,k,s \in \mathbb{Z}$ such that $k \geq 3$, $1 \leq d \leq k-1$, and $0 \leq s \leq k - 1$.
For every $\gamma > 0$, there exist $C>0$ and $n_0 \in \mathbb{N}$ such that the following holds for all $n \in k\mathbb{N} + s$ with $n \geq n_0$.
For every $k$-uniform $n$-vertex hypergraph~$\cH$ with $\delta_{d}(\cH) \geq \left(\overline{\mu_d}^{(s)}(k) + \gamma\right) \binom{n-d}{k-d}$, there exists a probability measure on the set of optimal matchings in~$\cH$ which is $(C/n^{k-1})$-spread.
\end{theorem}

\begin{theorem}\label{thm:main_rrs_fknp}
Let $k \geq 3$ be an integer. There exist $C>0$ and $n_0 \in \mathbb{N}$ such that the following holds for all integers $n \geq n_0$ divisible by $k$.
For every $k$-uniform $n$-vertex hypergraph $\cH$ with $\delta_{k-1}(\cH) \geq m_{k-1}(k,n)$, there exists a probability measure on the set of perfect matchings in $\cH$ which is $(C/n^{k-1})$-spread.
\end{theorem}

For a $k$-uniform $n$-vertex hypergraph $\cH$ with $\delta_{k-1}(\cH) \geq \delta n$ for some $\delta > 1/2$, there are some earlier results~\cite{FHM2021, FKS2016, GGJKO2021} on counting the number of perfect matchings in $\cH$.
Recently, Glock, Gould, Joos, K\"uhn, and Osthus~\cite{GGJKO2021} showed that $\cH$ has at least $\exp((1-1/k) n \log n - \Theta(n))$ perfect matchings, which is best possible up to an $\exp(\Theta(n))$ factor (this is also implicit in~\cite{FKS2016}), since this is also an upper bound for the number of perfect matchings in an $n$-vertex $k$-uniform complete hypergraph.
Very recently, Ferber, Hardiman, and Mond~\cite{FHM2021} sharpened the bound further by showing that $\cH$ has at least $(1-o(1))^n |\cM(\cK_n^k)| \delta^{n/k}$ perfect matchings, where $|\cM(\cK_n^k)|$ is the number of perfect matchings in an $n$-vertex complete $k$-uniform hypergraph.

As a corollary to Theorems~\ref{thm:main_general_fknp} and~\ref{thm:main_rrs_fknp}, we extend the above  results of~\cite{FKS2016, GGJKO2021} to $n$-vertex $k$-uniform hypergraphs with minimum $d$-degree at least $m_d(k,n) + o(n^{k-d})$ or minimum codegree at least $m_{k-1}(k,n)$. Our bounds are best possible up to an $\exp(\Theta(n))$ factor.


\begin{corollary}\label{cor:num_pm}
Let $k \geq 3$ be an integer and $\gamma \in (0,1)$. There exist $c>0$ and $n_0 \in \mathbb{N}$ such that the following holds for all integers $n \geq n_0$.
\begin{enumerate}[label=\upshape{(\roman*)}]
    \item\label{num_pm:general} For $d \in [k-1]$, if $n \equiv s \modu{k}$, then every $k$-uniform $n$-vertex hypergraph $\cH$ with $\delta_{d}(\cH) \geq \left(\overline{\mu_d}^{(s)}(k) + \gamma\right) \binom{n-d}{k-d}$ has at least $\exp((1-1/k) n \log n - cn)$ optimal matchings.
        
    \item If $k \mid n$, then every $k$-uniform $n$-vertex hypergraph $\cH$ with $\delta_{k-1}(\cH) \geq m_{k-1}(k,n)$, has at least $\exp((1-1/k) n \log n - cn)$ perfect matchings.
\end{enumerate}
\end{corollary}

    
\begin{proof}
Let $\cM(\cH)$ be the set of optimal matchings of $\cH$, and let ${\bf M}$ be a $(C/n^{k-1})$-spread random optimal matching of $\cH$ for some $C > 0$ given by Theorems~\ref{thm:main_general_fknp} or~\ref{thm:main_rrs_fknp}, where $C$ is a function of $k$ and $\gamma$ for (i), and a function of $k$ for (ii).

For each fixed matching $M \in \cM(\cH)$ which consists of $\lfloor n/k \rfloor$ edges $e_1 , \dots, e_{\lfloor n/k \rfloor}$, we have that $\pr{e_1 , \dots , e_{\lfloor n/k \rfloor} \in {\bf M}} = \pr{{\bf M} = M} \leq (C/n^{k-1})^{\lfloor n/k \rfloor}$ since ${\bf M}$ is $(C/n^{k-1})$-spread. Thus,
\begin{align*}
    1 = \sum_{M \in \cM(\cH)}\pr{{\bf M} = M} \leq |\cM(\cH)| (C/n^{k-1})^{\lfloor n/k \rfloor},   
\end{align*}
which implies $|\cM(\cH)| \geq (n^{k-1} / C)^{\lfloor n/k \rfloor} = \exp((1-1/k) n \log n - (\log C/k)n \pm k \log n)$, as desired.
\end{proof}

\begin{remark}
In independent work, Pham, Sah, Sawhney, and Simkin~\cite{PSSS22} also proved Theorem~\ref{thm:main_general_fknp} and its corollaries \cref{cor:num_pm}\ref{num_pm:general} and Theorem~\ref{thm:main_general}. In addition, they proved a robust version of the Hajnal--Szemer\'{e}di theorem~\cite{HS1970} regarding embedding a $K_r$-factor into an $n$-vertex graph with minimum degree at least $(1-1/r)n$ and a robust version of Koml\'{o}s, S\'{a}rk\"{o}zy, and Szemer\'{e}di's~\cite{KSS1995} proof of Bollob\'as' conjecture that an $n$-vertex graph with minimum degree at least $(1/2 + o(1))n$ contains any bounded-degree spanning tree.  Both of these results are also derived from stronger results concerning spread measures.  
\end{remark}

\subsection{Notation}
For $k \in \mathbb N$, we let $[k] \coloneqq \{1, \dots, k\}$.
For a $k$-uniform hypergraph $\cH$ and an edge $e \in \cH$, we often denote by $V(e)$ the set of vertices incident to $e$.
For any set $S \subseteq V(\mathcal{H})$, we denote by $\cH[S]$ the subgraph of $\mathcal{H}$ induced by $S$. 
For sets $S \subseteq V(\mathcal{H})$ and $\mathcal{S} \subseteq \binom{V(\cH)}{k-\s{S}}$, we denote by $N_\cH(S; \mathcal{S})$ the set of edges $e \in \cH$ such that $e = S \cup S'$ for some $S' \in \mathcal{S}$, and $d_\cH(S ; \mathcal{S}) \coloneqq |N_\cH(S; \mathcal{S})|$. We often omit $\mathcal{S}$ if $\mathcal{S} = \binom{V(\cH)}{k-\s{S}}$ (for example, we write $N_\cH(S)$ and $d_\cH(S)$).
If $\s{S} = k-1$ and $U \subseteq V(\cH)$, we abuse notation and write $d_\cH(S; U)$ for $d_\cH(S;\{\{u\} \colon u \in U\})$. Moreover, for $v \in V(\cH)$, we write $d_\cH(v; \mathcal{S})$ for $d_{\cH}(\{v\}; \mathcal{S})$. 
For disjoint subsets $W_1 , \dots , W_k \subseteq V(\cH)$, we denote by $e_\cH(W_1 , \dots , W_k)$ the number of edges $e \in \mathcal{H}$ with $|e \cap W_i| = 1$ for all $i \in [k]$.
We denote by $\overline{\cH}$ the \emph{complement} of a $k$-uniform hypergraph $\cH$ such that $V(\overline{\cH}) \coloneqq V(\cH)$ and  $\overline{\cH} \coloneqq \binom{V(\cH)}{k} \setminus \cH$.
We take all asymptotic notations $o(\cdot),O(\cdot),\Theta(\cdot),\omega(\cdot),\Omega(\cdot)$ to be as $n \to \infty$, and all of their leading coefficients may depend on parameters other than $n$. 
We say that an event $\mathcal{E}$ holds \emph{asymptotically almost surely \mbox{(a.a.s.)}} if $\pr{\mathcal{E}} = 1 - o(1)$ as $n \rightarrow \infty$. For real numbers $x$, $y$, $\alpha$, and $\beta$ with $\beta \geq 0$, we write $x =  (\alpha \pm \beta)y$ for $(\alpha - \beta)y \leq x \leq (\alpha + \beta)y$.
We sometimes state a result with a hierarchy of constants which are chosen from right to left.
If we state that the result holds whenever $a \ll b_1,\dots,b_t$, then this means that there exists a function $f \colon (0,1)^t \to (0,1)$ such that $f(b_1 , \dots , \widetilde{b_i} , \dots , b_t) \leq f(b_1 , \dots , b_i , \dots , b_t)$ for $0 < \widetilde{b_i} \leq b_i < 1$ for all $i \in [t]$ and the result holds for all real numbers $0 < a , b_1 , \dots , b_t < 1$ with $a \leq f(b_1 , \dots , b_t)$. If a reciprocal $1/m$ appears in such a hierarchy, we implicitly assume that~$m$ is a positive integer.
For a set $U$ and $p \in [0,1]$, a \emph{$p$-random subset} of $U$ is a random subset $U'$ of $U$ that contains each element of $U$ independently with probability $p$.

\subsection{Proof outline}\label{subsec:outline}

Here we briefly sketch the proofs of Theorems~\ref{thm:main_general_fknp} and \ref{thm:main_rrs_fknp}.  One key idea of both proofs is that we can use the weak hypergraph regularity lemma (Theorem~\ref{thm:Weak_HRL}) to find a distribution on \textit{almost} perfect matchings which has good spreadness.  The weak hypergraph regularity lemma gives us a \emph{reduced $k$-uniform hypergraph} $\cR$ such that almost all subsets of $V(\cR)$ of size $d$ have large $d$-degree. Using Lemma~\ref{lem:almostPM_general}, we can find an almost perfect matching $M_{\cR}$ of $\cR$ such that each edge of $M_{\cR}$ corresponds to a vertex-disjoint pseudorandom $k$-partite subhypergraph in which we can easily construct an almost perfect matching with spreadness.

To obtain a distribution on \textit{optimal} matchings, we use an approach inspired by the method of ``iterative absorption" (introduced in \cite{KKO2015, KK2013} and further developed in \cite{BGKLMO2020, BKLO2016, BKLOT2017, GKLMO2019, GKLO16, KKKMO22, KSSS22, SSS22}).  Our iterative-absorption approach, combined with the regularity lemma, allows us to `bootstrap' results on the existence of optimal matchings to construct well-spread distributions on optimal matchings.  In this approach, we choose a random partition $(U_1 , \dots , U_\ell)$ of the vertex-set of the $k$-uniform hypergraph $\cH$, which we call a \emph{vertex vortex} (Definition~\ref{def:vortex}), which \aas satisfies $|U_{i+1}| \sim |U_i|/2$, $|U_\ell| = O(n^{1/k})$, and some additional conditions on degrees of the vertices in $\cH[U_i]$ and $\cH[U_i , U_{i+1}]$. Using the regularity-lemma approach described above, we can find a well-spread distribution on matchings in $\cH[U_i]$ which cover almost all vertices.  Then we cover the leftover uncovered vertices in $U_i$ using edges which intersect $U_{i+1}$ in $k - 1$ vertices.  
By the degree conditions of the vertex vortex, there are many choices of such edges, so a random greedy approach yields a distribution with good enough spreadness.  After iterating this procedure $\ell - 1$ times, it suffices to find an optimal matching in the final subset $U_\ell$ (with a small subset of vertices deleted), in a deterministic way, since $|U_\ell| = O(n^{1/k})$ and
$\pr{e_1 , \dots , e_t \subseteq U_\ell} = (|U_\ell|/n)^{kt} = (O(1)/n^{k-1})^t$ for any $t$ disjoint edges $e_1 , \dots , e_t \in \cH$.

In the setting of Theorem~\ref{thm:main_general_fknp}, it is straightforward to find an optimal matching in the final step; the hypergraph induced on the remaining vertices will still be sufficiently dense.  For Theorem~\ref{thm:main_rrs_fknp}, we show that the result of R\"{o}dl, Ruci\'nski, and Szemer\'edi~\cite{Rodl2009} holds `robustly'.  Roughly speaking, it holds in the hypergraph induced by a random set of vertices, even after deleting a small proportion of the vertices.   
For this we must consider two cases according to whether the original hypergraph is close to being a `critical hypergraph' (see Definition~\ref{def:critical}) which has minimum codegree $m_{k-1}(k,n) - 1$ and no perfect matching.
If the original hypergraph $\cH$ is close to being a critical hypergraph, then we may choose an `atypical edge' among $\Omega(n^{k-1})$ candidates (Lemma~\ref{lem:manyatypical}) and delete its vertices in advance. This ensures that the subhypergraph of $\cH$ induced by the remaining vertices in $U_\ell$ will meet certain `divisibility' conditions and allow us to apply some technical results proved by R\"{o}dl, Ruci\'nski, and Szemer\'edi to find a perfect matching covering the remaining vertices of $U_\ell$.
Moreover, since there are $\Omega(n^{k-1})$ candidates for the atypical edge, we have the desired spreadness property for this edge.
In the second case, the original hypergraph $\cH$ is not close to being a critical hypergraph.  In this case, we prove that there are still many `absorbers' inside $U_\ell$ (Corollary~\ref{cor:M_count_in_U}), which we can use to build an `absorbing matching'.  As long as the vertices of the absorbing matching are not among those removed from $U_\ell$, we can transform an almost perfect matching (which covers most of the remaining vertices of $U_\ell$) into a perfect matching (i.e., one which covers all remaining vertices of $U_\ell$).

\section{Tools}

\subsection{Concentration inequalities}
We will use the following well-known version of the \emph{Chernoff bound}.

\begin{lemma}[Chernoff bound]\label{lem:chernoff}
If $X$ is the sum of mutually independent Bernoulli random variables, then for all $\delta \in [0,1]$,
\begin{equation*}
    \Prob{|X - \mathbb{E}[X]| \geq \delta \mathbb{E}[X]} \leq 2e^{-\delta^2 \mathbb{E}[X] / 3}.
\end{equation*}
\end{lemma}

\begin{definition}[Typical subset]\label{def:typical}
Let $V$ be a finite set, and let $\cF \subseteq 2^V$ be a collection of subsets of $V$. For $p,\eps \in [0,1]$, a subset $U \subseteq V$ is called \emph{$(p,\eps,\cF)$-typical} if the number of elements in $\cF$ contained in $U$ is $(1 \pm \eps) \sum_{S \in \cF}p^{|S|}$.
\end{definition}

We will use the following probabilistic lemma which follows from the Kim--Vu polynomial concentration theorem~\cite{Kim2000}.
For the proof, see \APPENDIX{\cref{{app:lemma-proofs}}}\NOTAPPENDIX{the appendix of the arxiv version of this paper}.

\begin{lemma} \label{lemma:prob_lemma}
    Let $1/n \ll 1/s, \beta, \eps < 1$ and $k \geq 2$. Let $V$ be a set of size $n$. Let $p = p(n) \in [0,1]$ such that $np \geq \eps n^\beta$. Let $\mathcal{F} \subseteq \binom{V}{s}$, and let $U$ be a $p$-random subset of~$V$. Then the following holds. 
    \begin{enumerate}[label = {\upshape{(\roman*})}, leftmargin= \widthof{a00000}]
        \item \label{typical_i} If $\s{\mathcal{F}} \geq \eps n^s (np)^{-1/2}$, then with probability at least $1 -  \exp(-n^{\beta/(10s)})$, 
        the set $U$ is $(p,\eps,\cF)$-typical.
        
        \item \label{typical_ii} If $\s{\mathcal{F}} \leq \eps n^s$, then with probability at least $1 -  \exp(-n^{\beta/(10s)})$, 
        the number of elements of $\mathcal{F}$ contained in $U$ is at most $2 \eps (np)^s$.
    \end{enumerate}
    
\end{lemma}

\subsection{Weak hypergraph regularity}

We now introduce the weak hypergraph regularity lemma, which states that any $k$-uniform hypergraph has a vertex partition into clusters $\{ V_i \}_{0 \leq i \leq t}$ so that almost all $k$-tuples of clusters induce $\eps$-regular subhypergraphs.
Since the notion of $\eps$-regularity is `weak', its proof is very similar to the graph version. 
Readers should not confuse the weak hypergraph regularity lemma with the Frieze--Kannan weak regularity lemma~\cite{FK1999}.

\begin{definition}[$\varepsilon$-regular $k$-tuple]
Let $\eps > 0$ and let $\cH$ be a $k$-uniform hypergraph. 
We say that a $k$-tuple $(V_1, \dots, V_k)$ of mutually disjoint subsets of $V(\cH)$ is \emph{$(d,\eps)$-regular} if
$e_\cH(W_1 , \dots , W_k) = (d \pm \eps)|W_1|\cdots|W_k|$
for every $W_1 \subseteq V_1$, \dots, and $W_k \subseteq V_k$ with $|W_1|\cdots|W_k| \geq \eps |V_1|\cdots|V_k|$. 
Moreover, we say that $(V_1, \dots, V_k)$ is \emph{$\eps$-regular} if it is $(d, \eps)$-regular for some $d > 0$.
\end{definition}

\begin{definition}[$\varepsilon$-regular partition]
Let $\eps > 0$, and let $\cH$ be a $k$-uniform hypergraph. A partition $(V_0 , V_1 , \dots , V_t)$ of $V(\cH)$ is called an \emph{$\varepsilon$-regular partition} if 
\begin{itemize}
    \item $|V_0| \leq \eps n$ and $|V_1| = \cdots = |V_t|$.
    \item For all but at most $\eps \binom{t}{k}$ $k$-sets $\{i_1, \dots, i_k \} \in \binom{[t]}{k}$, the tuple $(V_{i_1}, \dots, V_{i_k})$ is $\eps$-regular.
\end{itemize}
\end{definition}

\begin{definition}[Reduced hypergraph]
Let $\cH$ be a $k$-uniform hypergraph, and let $(V_0 , V_1 , \dots , V_t)$ be an $\eps$-regular partition of $V(\cH)$. 
The \emph{$(\gamma, \eps)$-reduced hypergraph} $\cR$ with respect to $(V_0 , V_1 , \dots , V_t)$ is the $t$-vertex $k$-uniform hypergraph with $V(\cR) = [t]$ and $\{i_1 , \dots , i_k \} \in \cR$ if and only if $(V_{i_1} , \dots , V_{i_k})$ is $\eps$-regular and $e_\cH(V_{i_1} , \dots , V_{i_k}) \geq \gamma|V_{i_1}|\cdots|V_{i_k}|$.
\end{definition}

\begin{theorem}[Weak hypergraph regularity lemma~\cite{Ch91, Fr92, St90}]\label{thm:Weak_HRL}
Let $1/n , 1/t_1 \ll \eps , 1/t_0 < 1$. For every $n$-vertex $k$-uniform hypergraph $\cH$, there exists an $\eps$-regular partition $(V_0 , \dots , V_t)$ of $V(\cH)$ such that $t_0 \leq t \leq t_1$.
\end{theorem}



The next lemma can be proved with a straightforward adaptation of the proof of~\cite[Proposition 16]{HS10}, so we defer the proof to \APPENDIX{\cref{app:lemma-proofs}}\NOTAPPENDIX{the appendix of the arxiv version of this paper}.


\begin{lemma}\label{lem:reduced_degree}
Let $1/n \ll \eta \ll 1/t \ll \eps \ll \gamma < c, 1/k \leq 1$ with $k \geq 3$ and $d \in [k-1]$.
Let $\cH$ be a $k$-uniform $n$-vertex hypergraph which satisfies the following.
\begin{itemize}
    \item All but at most $\eta n^{d}$ $d$-sets $S \in \binom{V(\cH)}{d}$ have $d$-degree at least $c\binom{n-d}{k-d}$.
    \item $\cH$ admits an $\eps$-regular partition $(V_0 , \dots , V_t)$.
\end{itemize}

Let $\cR$ be the $(\gamma/3,\eps)$-reduced hypergraph with respect to $(V_1 , \dots , V_t)$. Then all but at most $\eps^{1/2} \binom{t}{d}$ many $d$-sets $S \in \binom{[t]}{d}$ have $d$-degree at least $(c - \gamma) \binom{t-d}{k-d}$ in $\cR$.
\end{lemma}

\subsection{Almost perfect matchings}




For $1 \leq d \leq k-1$, recall that $m_d(k,n)$ is the minimum $D$ such that every $n$-vertex $k$-uniform hypergraph with minimum $d$-degree at least $D$ has an optimal matching. 
Let us define
\begin{equation*}
    \underline{\mu_d}(k) \coloneqq \liminf_{n \to \infty} \frac{m_d(k,n)}{\binom{n-d}{k-d}}.
\end{equation*}

Note that $\underline{\mu_d}(k) \leq \overline{\mu_d}^{(s)}(k)$ for $0 \leq s \leq k-1$, and $\underline{\mu_{k-1}}(k) = 1/k$, since $m_{k-1}(k,n) = n/2 - O(k)$ for large $n \in k \mathbb{N}$ and $m_{k-1}(k,n) = \lfloor n/k \rfloor$ for large $n \notin k \mathbb{N}$, as mentioned in Section~\ref{subsec:pm}.
A well-known lower bound on $\underline{\mu_d}(k)$ is $1 - (\frac{k-1}{k})^{k-d}$ (see~\cite[Construction 1.4]{zhao2016}).

Now we prove the following lemma which states that if almost all $d$-tuples satisfy the degree condition for an optimal matching then there exists an almost perfect matching. 
We also remark that there are also similar results on almost perfect matchings~\cite{GH2017, GHZ2019, Keevash2015}.

\begin{lemma}\label{lem:almostPM_general}
    Let $1/n \ll \eps_1 \ll \eps_2 \ll 1/k \leq 1/3$ with $1 \leq d \leq k-1$. Let $\cH$ be an $n$-vertex $k$-uniform hypergraph such that $d_\cH(S) \geq (\underline{\mu_d}(k) + \eps_2) \binom{n-d}{k-d}$ for all but at most $\eps_1 n^d$ many $S \in \binom{V(\cH)}{d}$. Then $\cH$ has a matching which covers all but at most $2 \eps_2 n$ vertices. 
\end{lemma}

To prove Lemma~\ref{lem:almostPM_general}, we use the following lemma~\cite[Lemma 3.4]{FK2022}.

\begin{lemma}[Ferber and Kwan \cite{FK2022}]\label{lem:mindeg_random}
Let $1/n \ll \delta \ll 1/m \ll \eps \ll c , 1/k < 1$. Let $1 \leq d \leq k-1$. 
Let $\cH$ be an $n$-vertex $k$-uniform hypergraph such that $d_\cH(S) \geq (c + \eps) \binom{n-d}{k-d}$ for all but at most $\delta n^d$ many $S \in \binom{V(\cH)}{d}$.
Let $U$ be a random subset of $V(\cH)$ of size $m$ uniformly chosen from $\binom{V(\cH)}{m}$. With probability at least $1 - m^d (\delta + e^{-\eps^3 m})$, we have $\delta_d (\cH[U]) \geq (c + \eps/2) \binom{m-d}{k-d}$.
\end{lemma}


\begin{proof}[Proof of Lemma~\ref{lem:almostPM_general}]
Let $1/n \ll \eps_1 \ll 1/m \ll \eps_2 \ll 1/k$ such that $m_d(k,m) \leq (\underline{\mu_d}(k) + \eps_2/2) \binom{m-d}{k-d}$. 
For $t \coloneqq \lfloor n/m \rfloor$, let $U_1 , \dots , U_t$ be $t$ disjoint random subsets of $V(\cH)$ of size $m$ such that each $U_i$ has a uniform random distribution from $\binom{V(\cH)}{m}$.
For each $i \in [t]$, let $U_i$ be \emph{bad} if $\delta_d (\cH[U_i]) < (\underline{\mu_d}(k) + \eps_2/2) \binom{m-d}{k-d}$, and otherwise \emph{good}. 
By Lemma~\ref{lem:mindeg_random}, for each $i \in [t]$, $\mathbb{P}(\text{$U_i$ is bad}) \leq m^d (\eps_1 + e^{-\eps_2^3 m}) < \eps_2^2$, so the expected number of bad $U_i$'s is at most $\eps_2^2 t$.
By Markov's inequality, with probability at least $1-\eps_2$, the number of bad $U_i$'s is at most $\eps_2 t$. Fix a choice of $U_1 , \dots , U_t$ for which this holds. 
For each of the good $U_i$'s, 
since $m_d(k,m) \leq (\underline{\mu_d}(k) + \eps_2/2) \binom{m-d}{k-d} \leq \delta_d (\cH[U_i])$,
there is an optimal matching $M_i$ of $\cH[U_i]$. Let $M \coloneqq \bigcup_{U_i:\:\text{good}} M_i$. Then
\begin{align*}
    |V(\cH) \setminus V(M)| &\leq \left |V(\cH) \setminus \bigcup_{i=1}^{t}U_i \right | + \sum_{U_i:\:\text{bad}} |U_i| + \sum_{U_i:\:\text{good}} |U_i \setminus V(M_i)| \\
    & \leq m-1 + \eps_2 t \cdot m + (k-1) \cdot (t - \eps_2 t)\\
    & \leq m + \eps_2 n + nk / m < 2 \eps_2 n,
\end{align*}
as desired.
\end{proof}

\section{Vortices and iterative absorption}

The main result of this section is Lemma~\ref{lem:spread_OM}, which essentially guarantees a $O(1 / n^{k - 1})$-spread measure on the set of optimal matchings in a $k$-uniform hypergraph $\cH$ in which a $O(1 / n^{1 - 1/k})$-random subset of vertices of $\cH$ induces a hypergraph with an optimal matching with high probability.  To prove this result, we use an `iterative absorption' approach.

\subsection{Vortices}

Recall from Section~\ref{subsec:outline} that a vertex vortex, formally defined below, is a sequence of vertex sets, which all induce relevant properties of the original hypergraph.  The first step in the proof of Lemma~\ref{lem:spread_OM} is to randomly partition the vertices of $\cH$, and this partition will be a vertex vortex with high probability.

\begin{definition}[Vertex vortex]\label{def:vortex}
    Let $k \geq 2$, and let $\mathcal{H}$ be a $k$-uniform hypergraph on $n$ vertices. For a positive integer $\ell$, a vector $\mathbf{p} = (p_1, \dots, p_\ell)$ of non-negative reals such that $\sum p_i =1$, an integer $d \in [k-1]$, and $\eps, \alpha_1, \alpha_2 > 0$, we say that a partition $(U_1, \dots, U_\ell)$ of $V(\mathcal{H})$ is an \emph{$(\alpha_1, \alpha_2, d, \eps, \mathbf{p})$-vortex for $\mathcal{H}$} if 
    \begin{enumerate}[label = {$(\text{V}\arabic*)$}, leftmargin= \widthof{V1000}]
        \item \label{V1} $\s{U_i} = (1 \pm \eps)p_i n$ for all $i \in [\ell]$,
        \item \label{V2} $d_{\cH[U_i]}(S) \geq (\alpha_1 -\eps)(p_i n)^{k-d}$ for all $i \in [\ell-1]$, and all but $\eps(p_i n)^{d}$ many $S \in \binom{U_i}{d}$, and 
        \item \label{V3} $d_{\mathcal{H}}(v; \binom{U_{i} \setminus \{ v \}}{k-1}) \geq (\alpha_2 - \eps) (p_{i} n)^{k-1}$ for all $i \in [\ell]$ and $v \in V(\cH)$.
    \end{enumerate}
\end{definition}

\begin{definition}[$(\alpha_1, \alpha_2, d, \eps)$-dense]
Let $k \geq 2$, let $d \in \{1 , \dots , k \}$, and let $\alpha_1 , \alpha_2 , \eps \in [0,1]$.
A $k$-uniform hypergraph $\cH$ on $n$ vertices is \emph{$(\alpha_1, \alpha_2, d, \eps)$-dense} if $d_\cH(S) \geq \alpha_1 n^{k-d}$ for all but $\eps n^{d}$ many $S \in \binom{V(\cH)}{d}$ and $d_\cH(v) \geq \alpha_2 n^{k-1}$ for all $v \in V(\cH)$.   
\end{definition}

The next lemma can be proved via a straightforward combination of Chernoff bounds and Lemma \ref{lemma:prob_lemma} with the union bound, so we \APPENDIX{defer it to the appendix}\NOTAPPENDIX{only include its proof in the appendix of the arxiv version of this paper}.

\begin{lemma}[Vortex existence lemma] \label{lem:vortex_existence}
    Let $1/n \ll \eps < \alpha_1, \alpha_2, 1/k < 1$ with $k \geq 3$ and $d \in [k - 1]$. Let $\cH$ be a $(\alpha_1, \alpha_2, d, \eps)$-dense $k$-uniform hypergraph on $n$ vertices. Let $\ell \coloneqq \left\lceil\frac{k-1}{k} \log_2(n)\right\rceil$, $C_\ell \coloneqq \sum_{i=1}^\ell 2^{-i}$, $p_i \coloneqq \frac{1}{C_\ell 2^{i}}$ for each $i \in [\ell]$, and $\mathbf{p} \coloneqq (p_1, \dots, p_\ell)$.
    Independently for each vertex $v \in V(\mathcal{H})$, let $X_v$ be a random variable with values in $[\ell]$ such that $\pr{X_v = i} = p_i$ for each $i \in [\ell]$. For each $i \in [\ell]$, let $U_i \coloneqq \{v \in V(\mathcal{H}) \colon X_v =i\}$. Then we have that \aas $(U_1, \dots, U_\ell)$ is an $(\alpha_1, \alpha_2, d, 2\eps, \mathbf{p})$-vortex for $\mathcal{H}$.
\end{lemma}

\subsection{Matchings inside vortex sets}\label{subsection:blowing-up}

To prove Lemma~\ref{lem:spread_OM}, we will find a well-spread measure on almost perfect matchings in each `level' of the vertex vortex using the weak hypergraph regularity lemma (Theorem~\ref{thm:Weak_HRL}).  The following lemma is key for this approach.

\begin{lemma}[Random matching in an $\eps$-regular $k$-tuple] \label{lem:random_M_in_k_tuple}
Let $1/n \ll \eps \ll d, 1/k < 1$. Let $\mathcal{H}$ be a $k$-partite $k$-uniform hypergraph with partition $(V_1, \dots, V_k)$ such that $\s{V_1} = \dots = \s{V_k} = n$ and $(V_1, \dots, V_k)$ is $\eps$-regular with density at least $d$.  Then there exists a $(1/(\eps^2n^{k-1}))$-spread probability measure on the set of matchings in~$\mathcal{H}$ which cover all but at most $2k\eps^{1/k}n$ vertices.
\end{lemma}
\begin{proof}
We define a probability measure on the set of matchings in $\cH$ which cover all but at most $2k\eps^{1/k}n$ vertices by randomly constructing a matching $M$ as follows.
Let $u_1 , \dots , u_n$ be an enumeration of the vertices in $V_1$. Let $M_0 \coloneqq \varnothing$, $W_0 \coloneqq \varnothing$, and $V_{i,0} \coloneqq V_i$ for each $i \in [k]$. We define $M_j \supseteq M_{j-1}$, $W_j \supseteq W_{j-1}$, $V_{i,j} \subseteq V_{i,j-1}$ for each $i \in [k]$ inductively to satisfy $|M_j| = n - |V_{2,j}|$, $|W_j| \leq |W_{j-1}| + 1$, and $|V_{2,j}| = \dots = |V_{k,j}| = n - j + |W_j|$ for each $j \geq 1$, until $|V_{2,j}| = \dots = |V_{k,j}| < 2 \eps^{1/k} n$.

Suppose $|V_{2,j-1}| = \dots = |V_{k,j-1}| \geq 2 \eps^{1/k} n$. We consider the following two cases:
\begin{itemize}
    \item If $e_\cH(u_j , V_{2,j-1} , \dots , V_{k,j-1}) < \eps^2 n^{k-1}$, then define $M_j \coloneqq M_{j-1}$, $W_j \coloneqq W_{j-1} \cup \{ u_j \}$, and $V_{i,j} \coloneqq V_{i,j-1}$ for each $2 \leq i \leq k$.
    
    \item Otherwise, if $e_\cH(u_j , V_{2,j-1} , \dots ,  V_{k,j-1}) \geq \eps^2 n^{k-1}$, then choose $(v_{2,j}, \dots, v_{k,j}) \in V_{2,j-1} \times \dots \times V_{k,j-1}$ uniformly at random so that $u_j v_{2,j} \dots v_{k,j} \in E_\cH(u_j , V_{2,j-1} , \dots, V_{k,j-1})$. 
    Define $M_j \coloneqq M_{j-1} \cup \{ u_j v_{2,j} \dots v_{k,j} \}$, $W_j \coloneqq W_{j-1}$, and $V_{i,j} \coloneqq V_{i,j-1} \setminus \{v_{i,j} \}$ for each $2 \leq i \leq k$.
\end{itemize}

Let $t \in [n]$ be the first index such that $|V_{2,t}| = \dots = |V_{k,t}| < 2 \eps^{1/k} n$. If such an index does not exist, then let $t \coloneqq n$. For either of the cases, we have $|V_{2,t}| = \dots = |V_{k,t}| > 2 \eps^{1/k} n - 1$.

Let $M \coloneqq M_t$. 
Since each edge of $\mathcal{H}$ is added to $M$ with probability at most $1/(\eps^2 n^{k-1})$ conditional on all other previous random choices,
the resulting measure is $1/(\eps^2 n^{k-1})$-spread\COMMENT{, since for every $s \geq 1$ and $e_1, \dots, e_s \in \cH$, assuming without loss of generality that $e_1, \dots, e_s$ is a matching where if $u_{i'} \in e_i \cap V_1$ and $u_{j'} \in e_j \cap V_1$ where $i' \leq j'$, then $i \leq j$, we have
\begin{equation*}
    \Prob{e_1, \dots, e_s \in M} = \prod_{i=1}^s \ProbCond{e_i \in M}{e_1, \dots, e_{i-1} \in M} \leq \left(\frac{1}{\eps^2 n^{k-1}}\right)^{s},
\end{equation*}
as required.}.


Now we aim to bound $|W_t|$. Indeed, for each $j \geq 1$ such that $u_j \in W_t$, we have
\begin{equation*}
    e_\cH(u_j , V_{2,t} , \dots, V_{k,t}) \leq e_\cH(u_j , V_{2,j-1} , \dots,  V_{t,j-1}) < \eps^2 n^{k-1},    
\end{equation*}
so $e_\cH(W_t , V_{2,t} , \dots,  V_{k,t}) < \eps^2 n^{k-1} |W_t| \leq \eps^2 n^k$. Since $|V_{2,t}|= \dots =|V_{k,t}| > 2\eps^{1/k} n - 1$, if $|W_t| > \eps^{1/k}n$, then $|W_t||V_{2,t}|\cdots|V_{k,t}| > \eps n^k = \eps |V_1|\cdots|V_k|$ while $e_\cH(W_t , V_{2,t} , \dots, V_{k,t}) < \eps^2 n^k < (d-\eps)|W_1|\s{V_{2,t}}\cdots \s{V_{k,t}}$, contradicting the assumption that $(V_1 , \dots , V_k)$ is $\eps$-regular with density at least $d$. Thus, $|W_t| \leq \eps^{1/k} n$. This also implies that $t < n$ and $|V_{2,t}| = \dots = |V_{k,t}| < 2 \eps^{1/k} n$; otherwise if $t = n$, then $|V_{2,n}| = \dots = |V_{k,n}| = n - (n - |W_t|) \leq \eps^{1/k} n$, which contradicts $|V_{2,t}| = \dots = |V_{k,t}| > 2 \eps^{1/k} n - 1$. 
Since $|M| = |M_t| = n - |V_{2,t}| \geq n - 2 \eps^{1/k} n$, the matching $M$ covers all but at most $2k \eps^{1/k} n$ vertices.
\end{proof}

The next lemma shows that we can find a well-spread measure on almost perfect matchings within a vortex set $U_i$.
\begin{lemma}[Random Matching inside a vortex set] \label{lem:M_in_votex_set}
Let $1/n \ll \delta \ll \eps_1 \ll \eps_2 \ll 1/k < 1$ with $k \geq 3$ and $d \in [k - 1]$. Let $\cH$ be a $k$-uniform hypergraph on $n$ vertices such that $d_\cH(S) \geq (\underline{\mu_d}(k) + \eps_2)\binom{n}{k-d}$ for all but at most $\eps_1 n^{d}$ many $d$-sets $S \in \binom{V(\cH)}{d}$. 
Then there exists a $(1/(\delta n^{k-1}))$-spread probability measure on the set of matchings in $\mathcal{H}$ which cover all but at most $2\eps_2 n$ vertices.
\end{lemma}
\begin{proof}
We define a probability measure on the set of matchings in $\cH$ which cover all but at most $2\eps_2 n$ vertices by randomly constructing a matching $M$ as follows.
Fix new constants $t_1$, $t_0$, $\eps$, and $\gamma$ such that $\eps_1 \ll 1/t_1 \ll 1/t_0 \ll \eps \ll \gamma \ll \eps_2$. By \cref{thm:Weak_HRL}, there exists an $\eps$-regular partition $(V_0, V_1, \dots, V_t)$ of $V(\cH)$ with $t_0 \leq t \leq t_1$. 
Let $\cR$ be the $(\gamma/3, \eps)$-reduced graph with respect to $(V_0, V_1, \dots, V_t)$.
By Lemma~\ref{lem:reduced_degree}, all but at most $\eps^{1/2} \binom{t}{d}$ many $d$-sets $S \in \binom{[t]}{d}$ satisfy $d_{\cR}(S) \geq (\underline{\mu_d}(k) + \eps_2/2)\binom{t}{k-d}$. Thus, by Lemma~\ref{lem:almostPM_general}, $\cR$ has a matching $M_{\cR}$ covering all but at most $\eps_2 t$ vertices.
Let $n_* \coloneqq \frac{n - \s{V_0}}{t} \geq (1- \eps) \frac{n}{t}$.
For each $S = \{i_1, \dots, i_k \} \in M_\cR$, by \cref{lem:random_M_in_k_tuple}, there exists a probability measure $\nu_S$ on the set of matchings in $\cH[V_{i_1}, \dots, V_{i_k}]$ that cover all but at most $2k\eps^{1/k}n_*$ of the vertices in $V_{i_1} \cup \dots \cup V_{i_k}$ 
that is $(1/(\eps^2n_*^{k-1}))$-spread.
Choose $M = \bigcup_{S \in M_\cR} M_S$ where each $M_S$ is chosen independently at random according to $\nu_S$.
Since
\[\label{eq:spreadness-blowup}
\frac{1}{\eps^2 n_*^{k-1}} \leq \frac{t^{k-1}}{\eps^2 (1-\eps)^{k-1}n^{k-1}} \leq \frac{1}{\delta n^{k-1}},
\]
the probability measure on $M$ is $(1/(\delta n^{k-1}))$-spread\COMMENT{, since for every $s \geq 1$ and $e, \dots, e_s \in \cH$, letting $X_S \coloneqq \cH[V_{i_1}, \dots, V_{i_k}] \cap \{e_1, \dots, e_s\}$ for every $S = \{i_1, \dots, i_k\} \in \cR$, we have
\begin{equation*}
    \Prob{e_1, \dots, e_s \in M} \leq \prod_{S \in \cM_R} \Prob{X_S \subseteq M_S} \leq \prod_{S \in \cM_R}\left(\frac{1}{\eps^2 n_*^{k-1}} \right)^{|X_S|} \leq \left(\frac{1}{\delta n^{k-1}}\right)^{s}.
\end{equation*}
}.
Moreover, $M$ covers all but at most
\[
\eps n + 2k\eps^{1/k} n_* \cdot \frac{t}{k} + \eps_2 t \cdot n_* \leq 2\eps_2 n
\]
vertices of $\cH$, as desired.
\end{proof}

\subsection{Covering down}

The following lemma will be used to cover the vertices in some vertex set $U_i$ by edges whose other vertices lie in $U_{i+1}$, i.e., we will apply it with $A$ playing the role of the set of uncovered vertices in $U_i$ and $B$ that of $U_{i+1}$.

\begin{lemma}[Cover-down lemma] \label{lem:M_cover_down}
Suppose $1/n \ll \eta \ll \delta \ll c, 1/k < 1$. Let $\cH$ be a $k$-uniform hypergraph on $n$ vertices, and let $(A,B)$ be a partition of $V(\cH)$ such that $\s{A} \leq \eta n$ and for each $v \in A$, $d_{\cH}(v; \binom{B}{k-1}) \geq cn^{k-1}$. Then there exists a $(1 / (\delta n^{k-1}))$-spread probability measure on the set of matchings $M$ in $\cH$ of size $\s{A}$ that cover~$A$ and satisfy $\s{e \cap A} = 1$ for each $e \in M$.
\end{lemma}
\begin{proof}
We define a probability measure on the set of matchings $M$ in $\cH$ of size $\s{A}$ that cover $A$ and $\s{e \cap A} = 1$ for each $e \in M$ by randomly constructing a matching $M$ as follows.
Let $m \coloneqq \s{A}$, and let $u_1, \dots, u_m$ be an enumeration of the vertices in $A$. Independently for each $i = 1, \dots, m$ in order, choose $S_i \in \binom{B \setminus \bigcup_{j=1}^{i-1}S_j}{k-1}$ such that $e_i \coloneqq u_i \cup S_i \in \cH$ uniformly at random. Let $M \coloneqq \{e_i \colon i \in [m]\}$. Note that for each $i \in [m]$, we have
\begin{align*}
d_\cH\left(u_i; \binom{B\setminus \bigcup_{j=1}^{i-1}S_j}{k-1}\right) &\geq d_\cH\left(u_i; \binom{B}{k-1}\right) - (k-1)|A||B|^{k-2}\\
&\geq cn^{k-1} - (k-1) \eta n^{k-1} \geq \frac{cn^{k-1}}{2}.    
\end{align*}
Hence each edge $e \in \cH$ is added to $M$ with probability at most ${2}/({cn^{k-1}}) \leq {1}/({\delta n^{k-1}})$ irrespective of all other random choices. It follows that the resulting measure is $({1}/({\delta n^{k-1}}))$-spread\COMMENT{, since for every $s \geq 1$ and $e_1, \dots, e_s \in \cH$, assuming without loss of generality that $e_1, \dots, e_s$ is a matching where if $u_{i'} \in e_i \cap A$ and $u_{j'} \in e_j \cap A$ where $i' \leq j'$, then $i \leq j$, we have
\begin{equation*}
    \Prob{e_1, \dots, e_s \in M} = \prod_{i=1}^s \ProbCond{e_i \in M}{e_1, \dots, e_{i-1} \in M} \leq \left(\frac{1}{\delta n^{k-1}}\right)^{s},
\end{equation*}
as required.}.
\end{proof}

\subsection{Spreadness of random matchings}

Given a vertex vortex $(U_1, \dots, U_\ell)$ of a hypergraph $\cH$, we can iteratively apply Lemmas~\ref{lem:M_in_votex_set} and \ref{lem:M_cover_down} $\ell - 1$ times to obtain a well-spread measure on matchings of $\cH$ which cover all vertices of $\cH$ not in $U_\ell$.  However, edges in `lower levels' (i.e.\ $U_i$ for $i$ close to $\ell$) of the vortex may be more likely to appear in this matching than edges in `higher levels' (i.e.\ $U_i$ for $i$ close to $1$), so we need to introduce the following `weighted' version of spreadness.  Since edges are less likely to appear in the lower levels of a random vortex, the spreadness `balances'.

\begin{definition}
Let $\cH$ be a $k$-uniform hypergraph, and let $\mathbf q = \left(q_e\right)_{e\in\cH}$ where $q_e \in [0, 1]$ for every $e\in\cH$.  A probability measure $\nu$ on the set of matchings in $\cH$ is \textit{$\mathbf q$-spread} if for every $S \subseteq \cH$, we have
\begin{equation*}
    \Prob{S \subseteq M} \leq \prod_{e\in S}q_e,
\end{equation*}
where $M$ is chosen at random according to $\nu$.
\end{definition}

Given a vertex vortex $(U_1, \dots, U_\ell)$, the following lemma provides a $\mathbf q$-spread measure for appropriately chosen $\mathbf q$ on matchings which cover all vertices not in $U_\ell$.  For technical reasons discussed later, we also need to control the parity of these matchings, we need these matchings to avoid a small `protected' set $U_* \subseteq U_\ell$, and we need that these matchings do not cover too many vertices of $U_\ell$.

\begin{lemma} \label{lem:M_given_vortex}
Suppose $1/n \ll \delta \ll \eps_* \ll \eps \ll c, 1/k < 1$ with $k \geq 3$, and $d \in [k-1]$. 
Let $\ell \coloneqq \left\lceil\frac{k-1}{k} \log_2(n)\right\rceil$, $C_\ell \coloneqq \sum_{i=1}^\ell 2^{-i}$, $p_i \coloneqq \frac{1}{C_\ell 2^{i}}$ for each $i \in [\ell]$, and $\mathbf{p} \coloneqq (p_1, \dots, p_\ell)$. 
Let $\cH$ be a $k$-uniform hypergraph on $n$ vertices, and let $(U_1, \dots, U_\ell)$ be an $\left(\frac{\underline{\mu_d}(k) + 4\eps}{(k-d)!}, c, d, \eps_*, \mathbf{p}\right)$-vortex for $\cH$. Let $U_* \subseteq U_\ell$ with $\s{U_*} \leq \eps \s{U_\ell}$ and $s \in \{0,1\}$. Let $\mathbf{q} \coloneqq \left(q_e\right)_{e\in\cH}$, where for each $e \in \cH$, 
\begin{equation*}
    q_e \coloneqq \left\{\begin{array}{l l}
    \frac{1}{\delta (p_i n)^{k-1}} & \text{ if $e \subseteq U_i$ for some $i \in [\ell -1]$},\\
    \frac{1}{\delta (p_{i+1} n)^{k-1}} & \text{ if $e \subseteq U_i \cup U_{i+1}$ and $\s{e \cap U_i} = 1$ for some $i \in [\ell -1]$}, \\
    1 & \text{ if $e \subseteq U_\ell$, and}\\ 
     0 & \text{ otherwise.}
    \end{array}\right.
\end{equation*}
Then there exists a $\mathbf{q}$-spread probability measure on the set of matchings $M$ in $\cH$ which satisfy $\s{M} \equiv s \modu{2}$, $U_* \subseteq V(\cH) \setminus V(M) \subseteq U_\ell$, and $\s{V(M) \cap U_\ell} \leq \eps^2 p_\ell n$.
\end{lemma}
\begin{proof}
Fix a new constant $\delta_*$ such that $\delta \ll \delta_* \ll \eps_*$.
We prove by induction on $j$ that for each $j$ such that $0 \leq j \leq \ell -1$, there exists a $\mathbf{q}|_{\cH[U_1 \cup \dots \cup  U_{j+1}]}$-spread probability measure $\nu_j$ on the set of matchings $M$ in $\cH[U_1 \cup \dots \cup U_{j+1}]$
which satisfy $U_* \subseteq V(\cH) \setminus V(M) \subseteq U_{j+1} \cup \dots \cup U_\ell$, $\s{V(M) \cap U_{j+1}} \leq \eps^2p_{j+1} n / 2$, and $e \not\subseteq U_{j+1} \cup \dots \cup U_\ell$ for each $e \in M$.

To see how the lemma follows from the existence of such $\nu_{\ell-1}$, note that $\nu_{\ell-1}$ is supported on the set of matchings $M$ of $\cH$ which satisfy $U_* \subseteq V(\cH) \setminus V(M) \subseteq U_\ell$, $|V(M) \cap U_\ell| \leq \eps^2 p_\ell n / 2$, and $e \not\subseteq U_\ell$ for each $e \in M$. 
If $|M| \equiv s \modu{2}$, then let $\nu(\{ M \}) \coloneqq \nu_{\ell-1} ( \{ M \} )$. Otherwise, we choose an arbitrary edge $e_M \in \cH[U_\ell \setminus V(M)]$ and let $\nu(\{ M \cup \{ e_M \} \}) \coloneqq \nu_{\ell-1} ( \{ M \} )$. 
Since $e \not\subseteq U_\ell$ for each $e \in M$ and $q_f = 1$ for each $f \in \cH[U_\ell]$, $\nu$ is a well-defined $\mathbf{q}$-spread probability measure on the set of matchings $N$ in $\cH$ which satisfy $\s{N} \equiv s \modu{2}$, $U_* \subseteq V(\cH) \setminus V(N) \subseteq U_\ell$, $\s{V(N) \cap U_{\ell}} \leq \eps^2p_{\ell} n$.

We define the desired probability measure by randomly constructing a matching $M$ as follows.
For $j = 0$ the statement trivially holds for $M = \varnothing$. Now let $j \geq 1$, and let $\nu_{j-1}$ be a $\mathbf{q}|_{\cH[U_1 \cup \dots \cup U_j]}$-spread probability measure on the set of matchings $M_*$ in $\cH[U_1 \cup \dots \cup U_{j}]$ 
which satisfy $U_* \subseteq V(\cH) \setminus V(M_*) \subseteq U_{j} \cup \dots \cup U_\ell$, $\s{V(M_*) \cap U_{j}} \leq \eps^2p_{j} n / 2$, and $e \not\subseteq U_j \cup \dots \cup U_\ell$ for each $e \in M_*$. 
Let $U'_j$ be the set of vertices in $U_j$ that are not covered by $M_*$. 
Since $\s{U_j} = (1 \pm \eps_*) p_j n$ and by the fact that $M_*$ covers at most $\eps^2 p_j n / 2$ vertices in $U_j$, we have $|U'_j| = (1 \pm 2\eps^2) p_jn$. 

Since $(U_1, \dots, U_\ell)$ is an $\left(\frac{\underline{\mu_d}(k) + 4\eps}{(k-d)!}, c, d, \eps_*, \mathbf{p}\right)$-vortex for $\cH$, for all but $\eps_* (p_j n)^{d}$ many $S \in \binom{U_j}{d}$, we have that $d_{\cH[U_j]}(S) \geq \left(\frac{\underline{\mu_d}(k) + 4\eps}{(k-d)!} - \eps_*\right)(p_j n)^{k-d} \geq \frac{\underline{\mu_d}(k) + 3\eps}{(k-d)!}(p_j n)^{k-d}$. It follows that for all but at most $\eps_* (p_j n)^{d} \leq \frac{\eps_* (p_j n)^{d}}{\s{U_j'}^{d}} |U_j'|^{d} \leq \frac{\eps_*}{(1-2\eps^2)^{d}} |U_j'|^{d} \leq 2 \eps_* |U_j'|^{d}$ many $S \in \binom{U_j'}{d}$, we have 
\begin{align*}
    d_{\cH[U_j']}(S) &\geq \frac{\underline{\mu_d}(k) + 3\eps}{(k-d)!}(p_j n)^{k-d} - 2\eps^2 (p_j n)^{k-d} \geq \frac{\underline{\mu_d}(k) + 2\eps}{(k-d)!}(p_j n)^{k-d} \\ &\geq \frac{\underline{\mu_d}(k) + 2\eps}{(k-d)!} \left(\frac{|U_j'|}{1+2\eps^2}\right)^{k-d} \geq (\underline{\mu_d}(k) + \eps) \binom{|U_j'|}{k-d}. 
\end{align*}

By applying \cref{lem:M_in_votex_set} with $|U'_j|$, $\delta_*$, $2\eps_*$, $\eps^4/2$, $k$, $d$, $\cH[U_j']$ playing the roles of $n$, $\delta$, $\eps_1$, $\eps_2$, $k$, $d$, $\cH$ (noting that each $e \in \cH[U'_j]$ satisfies $q_e = \frac{1}{\delta (p_j n)^{k-1}} \geq \frac{1}{\delta_* |U'_j|^{k-1}}$), there exists a $\mathbf q|_{\cH[U'_j]}$-spread probability measure $\nu_j'$ on the set of matchings $M_j$ in $\cH[U'_j]$ that cover all but at most $\eps^4 |U'_j|$ vertices of $U'_j$. 
Let $A$ be the set of vertices in $U_j'$ not covered by $M_j$, and let $B \coloneqq U_{j+1} \setminus U_*$. 


Let $\cG \coloneqq \cH[A \cup B]$. Note that $(1- 2\eps) p_{j+1} n \leq |V(\cG)| = \s{A} + \s{B} \leq \eps^4 (1 + \eps_*) p_j n + (1+ \eps_*) p_{j+1}n$. 
Using the fact that $p_{j} = 2p_{j+1}$, we have $|V(\cG)| = (1 \pm 2\eps) p_{j+1}n$ and $\s{A} \leq \eps^3 |V(\cG)|$. By \labelcref{V3}, we have for each $v \in A$, $d_\cG(v; \binom{B}{k-1}) \geq (c-\eps_*) (p_{j+1} n)^{k-1} - \s{U_*}\s{U_{j+1}}^{k-2} \geq (c - 3\eps) (p_{j+1} n)^{k-1} \geq \frac{c}{2} |V(\cG)|^{k-1}$. 
By applying \cref{lem:M_cover_down} with $|V(\cG)|$, $\eps^2$, $\delta_*$, $\frac{c}{2}$, $k$, $\cG$ playing the roles of $n$, $\eta$, $\delta$, $c$, $k$, $\cH$ (noting that each $e \in \cG$ with $\s{e \cap A} =1$ satisfies $q_e = \frac{1}{\delta (p_{j+1} n)^{k-1}} \geq \frac{1}{\delta_* |V(\cG)|^{k-1}}$), there exists a $\mathbf q|_{\cG}$-spread probability measure $\nu'_j$ on the set of matchings $M'_j$ in $\cG$ that cover $A$ and $(k-1) \s{A}$ vertices in $B$. 
Note that $M'_j$ covers $(k-1)\s{A} \leq (k-1)\eps^3 |V(\cG)| \leq \eps^2 p_{j+1}n / 2$ vertices in $U_{j+1}$.



It follows that we have randomly constructed the desired matching $M \coloneqq M_* \cup {M}_j \cup M'_j$, and the resulting measure $\nu_{j}$ is $\mathbf q|_{\cH[U_1 \cup \dots \cup U_{j+1}]}$-spread, 
since for every $S\subseteq \cH$, we have
\begin{align*}
    \Prob{S \subseteq  M} &= \Prob{S \cap \cH[U_1 \cup \dots \cup U_j] \subseteq M_*}\ProbCond{S \cap \cH[U'_j] \subseteq {M}_j}{M_*}
    \ProbCond{S \cap \cG \subseteq M'_j}{M_*, {M}_j}\\ 
    & \leq \prod_{e \in S}q_e.
\end{align*}

It is straightforward to check that $U_* \subseteq V(\cH) \setminus V(M) \subseteq U_{j+1} \cup \dots \cup U_\ell$, $\s{V(M) \cap U_{j+1}} \leq \eps^2p_{j+1} n / 2$, and $e \not\subseteq U_{j+1} \cup \dots \cup U_\ell$ for each $e \in M$.
\end{proof}

Since we apply Lemma~\ref{lem:M_given_vortex} to a \textit{random} vortex $(U_1, \dots, U_\ell)$, we can extend the random matching $M$ in a deterministic way without affecting the spreadness of the resulting measure.  It is of course crucial that there is at least one way to extend $M$ to an optimal matching (OM in this definition stands for `optimal matching'), which is captured by the following definition.

\begin{definition}[OM-stability]\label{def:om_stable}
For a $k$-uniform hypergraph $\cH$, a spanning subhypergraph $\cH'$ of $\cH$, and $\eps > 0$, we say that $U \subseteq V(\cH)$ is \emph{$(\cH',\eps)$-OM-stable for $\cH$} if there exists $U_* \subseteq U$ with $\s{U_*} \leq \eps \s{U}$ and $s \in \{0,1\}$ such that for any matching $M$ in $\cH'$ with $\s{M} \equiv s \modu{2}$, $U_* \subseteq V(\cH) \setminus V(M) \subseteq U$, and $\s{V(M) \cap U} \leq \eps \s{U}$, we have that $\cH - V(M)$ contains an optimal matching.
\end{definition}

In this definition, we only consider matchings inside some subhypergraph $\cH' \subseteq \cH$ in order to maintain some divisibility conditions in the critical case in the proof of Theorem~\ref{thm:main_rrs_fknp}.  We also only consider matchings of a certain parity for similar reasons. The set $U_*$ can be viewed as a set of `protected' vertices.
In the non-critical case of Theorem~\ref{thm:main_rrs_fknp}, we will find an `absorbing matching' on these vertices which can absorb a small set of leftover vertices.
This is discussed further in Section~\ref{sec:nonextremal}.




\begin{lemma} \label{lem:spread_OM}
Let $1/n \ll \delta \ll \delta_* \ll \eps_* \ll \eps \ll c, 1/k < 1$ with $k \geq 3$, and $d \in[k-1]$. 
Let $\cH$ be a $k$-uniform hypergraph on $n$ vertices and $\cH'$ a $\left(\frac{\underline{\mu_d}(k) + 4 \eps}{(k-d)!}, c, d, \eps_*\right)$-dense spanning subhypergraph of $\cH$. Let $\ell \coloneqq \left\lceil\frac{k-1}{k} \log_2(n)\right\rceil$, $C_\ell \coloneqq \sum_{i=1}^\ell 2^{-i}$, and $p_\ell \coloneqq \frac{1}{C_\ell 2^{\ell}}$. Suppose that a $p_\ell$-random subset of $V(\cH)$ is $(\cH', \eps)$-OM-stable for $\cH$ with probability at least $\delta_*$. 
Then there exists a probability measure on the set of optimal matchings of $\cH$ that is $\frac{1}{\delta n^{k-1}}$-spread.
\end{lemma}
\begin{proof}
For each $i \in [\ell - 1]$, let $p_i \coloneqq \frac{1}{C_\ell 2^{i}}$, and let $\mathbf{p} \coloneqq (p_1, \dots, p_\ell)$.
Independently for each vertex $v \in V(\mathcal{H})$, let $X_v$ be a random variable with values in $[\ell]$ such that $\pr{X_v = i} = p_i$ for each $i \in [\ell]$. For each $i \in [\ell]$, let $U_i \coloneqq \{v \in V(\mathcal{H}) \colon X_v =i\}$. Let $\mathcal{E}_1$ be the event that $(U_1, \dots, U_\ell)$ is a $\left(\frac{\underline{\mu_d}(k) + 4 \eps}{(k-d)!}, c, d, 2\eps_*, \mathbf{p}\right)$-vortex for $\cH'$, and let $\mathcal{E}_2$ be the event that $U_\ell$ is $(\cH',\eps)$-OM-stable for $\cH$. By \cref{lem:vortex_existence}, $\pr{\mathcal{E}_1} \geq 1 - \delta_*/2$, and by assumption, $\pr{\mathcal{E}_2} \geq \delta_*$. Hence, $\pr{\mathcal{E}_1 \cap \mathcal{E}_2} \geq \delta_*/2$.

Suppose that the outcome of $X_v, v \in V(\cH)$ is such that $\mathcal{E}_1 \cap \mathcal{E}_2$ holds. Since $U_\ell$ is $(\cH',\eps)$-OM-stable for $\cH$, there exists $U_* \subseteq U_\ell$ with $\s{U_*} \leq \eps \s{U_\ell}$ and $s \in \{0,1\}$ such that for any matching $M$ of $\cH'$ with $\s{M} = s \modu{2}$, $U_* \subseteq V(\cH) \setminus V(M) \subseteq U_\ell$, and $\s{V(M) \cap U} \leq \eps \s{U_\ell}$, we have that $\cH - V(M)$ contains an optimal matching. By \cref{lem:M_given_vortex} with $n$, $\delta_*$, $2\eps_*$, $\eps$, $c$, $k$, $d$, $\cH'$, $U_*$ playing the roles of $n$, $\delta$, $\eps_*$, $\eps$, $c$, $k$, $d$, $\cH$, $U_*$, there is a 
$\mathbf q$-spread probability measure $\nu_*$ on the set of matchings $M_*$ in $\cH'$ which satisfy $|M_*| \equiv s \modu{2}$, $U_* \subseteq V(\cH) \setminus V(M_*) \subseteq U_\ell$, and $|V(M_*) \cap U_\ell| \leq \eps^2 p_\ell n$, where $\mathbf q$ is as defined in \cref{lem:M_given_vortex}. 
Since $\mathcal{E}_2$ holds, we can complete the matching $M_*$ to an optimal matching $M$ of $\cH$.
Thus, conditional on the event $\mathcal{E}_1 \cap \mathcal{E}_2$, this procedure defines a probability measure on the set of optimal matchings $M$ in $\cH$.
(For each optimal matching $M$, the probability of $M$ appearing is given by the probability that this procedure outputs $M$. Note that for fixed $M$, there may be several different ways of arriving at output $M$ via this procedure.)

We claim that the resulting measure is $2q / \delta_*$-spread, where $q \coloneqq 4 / (\delta_* n^{k-1})$.
To that end, let $s \geq 1$, and let $e_1, \dots, e_s$ be distinct edges of $\cH$. We show that $\pr{e_1, \dots, e_s \in M} \leq 2q^s / \delta_* \leq (2q / \delta_*)^s$. If the edges $e_1, \dots, e_s$ do not form a matching in $\cH$, then clearly $\pr{e_1, \dots, e_s \in M} = 0$ as $M$ is a matching, so we may assume that the edges $e_1, \dots, e_s$ form a matching in $\cH$.  Let $\cP$ denote the set of partitions $(S_1, S'_1, \dots, S_{\ell - 1}, S'_{\ell - 1}, S_\ell)$ of $\{e_1, \dots, e_s\}$ into $2\ell - 1$ parts.  For each $P = (S_1, S'_1, \dots, S_{\ell - 1}, S'_{\ell - 1}, S_\ell) \in \cP$, let $\mathcal E_P$ be the event that
\begin{itemize}
    \item $e \subseteq U_i$ for all $i \in [\ell]$ and $e \in S_i$ and
    \item $|e \cap U_i| = 1$ and $|e \cap U_{i+1}| = k - 1$ for all $i \in [\ell - 1]$ and $e \in S_i'$.
\end{itemize}
Now
\begin{equation*}
    \Prob{e_1, \dots, e_s \in M} = \sum_{P \in \cP}\ProbCond{e_1, \dots, e_s \in M}{\curlyE_P}\ProbCond{\curlyE_P}{\curlyE_1 \cap \curlyE_2}.
\end{equation*}
Since $\{e_1, \dots, e_s\}$ is a matching, for every $P = (S_1, S'_1, \dots, S_{\ell - 1}, S'_{\ell - 1}, S_\ell) \in \cP$, we have
\begin{equation*}
   \ProbCond{\curlyE_P}{\curlyE_1 \cap \curlyE_2} \leq \frac{\Prob{\curlyE_P}}{\Prob{\curlyE_1 \cap \curlyE_2}} \leq \frac{2}{\delta_*} \prod_{i=1}^{\ell} p_i^{k|S_i|} \prod_{i=1}^{\ell-1} \left(p_i p_{i+1}^{k-1}\right)^{|S'_i|},
\end{equation*}
and by \cref{lem:M_given_vortex},
\begin{equation*}
    \ProbCond{e_1, \dots, e_s \in M}{\curlyE_P} \leq \prod_{i=1}^\ell q_i^{|S_i|} \prod_{i=1}^{\ell-1} {{q_i}'}^{|S'_i|},
\end{equation*}
where $q_i \coloneqq 1 / (\delta_* (p_i n)^{k - 1})$ and 
${q_i}' \coloneqq 1 / (\delta_* (p_{i+1} n)^{k - 1})$ for $i \in [\ell - 1]$ and $q_\ell \coloneqq 1$.
Since $q = 4 / (\delta_* n^{k-1})$, for all $i \in [\ell - 1]$,
\begin{equation*}
    q_i p_i^k = q'_i p_i p_{i+1}^{k - 1} = \frac{q p_i}{4},
\end{equation*}
and since $1 / n \ll \delta_* \ll 1/k$,\COMMENT{Note that $p_\ell = 1 / (C_\ell 2^\ell) = 1 / (C_\ell 2^{\log_2 n^{(k-1)/k} \pm 1}) \leq 2 / (C_\ell n^{(k-1)/k})$ and $C_\ell = 1 - 2^{-\ell} \in [0.99, 1]$.}
\begin{equation*}
    q_\ell p_\ell^k = p_\ell^k \leq \left ( \frac{2}{C_\ell n^{\frac{k-1}{k}}} \right)^k \leq \frac{2}{\delta_* n^{k-1}} = \frac{q}{2}.
\end{equation*}
Therefore, combining the five equations above, we have
\begin{align*}
    \Prob{e_1, \dots, e_s \in M} &\leq \frac{2}{\delta_*}\sum_{(S_1, S'_1, \dots, S_{\ell - 1}, S'_{\ell - 1}, S_\ell) \in \cP} \left(\frac{q}{2}\right)^{|S_\ell|}\prod_{i=1}^{\ell - 1} \left(\frac{qp_i}{4}\right)^{|S_i|} \left(\frac{qp_i}{4}\right)^{|S'_i|} \\
    &= \frac{2}{\delta_*} q^{s}\left(\frac{1}{2} + \frac{p_1}{4} + \frac{p_1}{4} + \cdots + \frac{p_{\ell-1}}{4} + \frac{p_{\ell-1}}{4}\right)^s \leq \frac{2}{\delta_*} q^{s},
\end{align*}
so our measure is $2q / \delta_*$-spread, as claimed.  Since $\delta \ll \delta_*$, the measure is also $1 / (\delta n^{k - 1})$-spread, as desired.
\end{proof}

\section{OM-stability}\label{section:cases}

In this section, we prove Theorem~\ref{thm:main_general_fknp}, and subject to some lemmas proved in later sections, we also prove Theorem~\ref{thm:main_rrs_fknp}. 
Lemma~\ref{lem:spread_OM} essentially reduces these proofs to the problem of proving the hypergraphs under consideration are OM-stable.

\subsection{\texorpdfstring{Proof of Theorem~\ref{thm:main_general_fknp}}{Proof of Theorem~\ref{thm:main_general_fknp}}}

Together with~\cref{lem:spread_OM}, the next lemma implies spreadness of optimal matchings in the case when we have minimum $d$-degree at least $(\overline{\mu_d}^{(s)}(k) + o(1)) \binom{n-d}{k-d}$.

\begin{lemma}\label{lem:d_deg_OM-stable}
    Let $1/n \ll \eps \ll \gamma \ll 1/k \leq 1/3$ and $d \in [k-1]$. Let $\cH$ be a $k$-uniform hypergraph on $n$ vertices with $\delta_d(\cH) \geq \left(\overline{\mu_d}^{(s)}(k) + \gamma\right) \binom{n-d}{k-d}$, where $n \equiv s\:({\rm mod}\:k)$ for $0 \leq s \leq k-1$. 
    Let $\ell \coloneqq \left\lceil\frac{k-1}{k} \log_2(n)\right\rceil$, $C_\ell \coloneqq \sum_{i=1}^\ell 2^{-i}$, and $p_\ell \coloneqq \frac{1}{C_\ell 2^{\ell}}$. Then \aas a $p_\ell$-random subset of $V(\cH)$ is $(\cH, \eps)$-OM-stable for $\cH$.
\end{lemma}
\begin{proof}
By the definition of $\overline{\mu_d}^{(s)}(k)$, there exists $n_0 \in \mathbb{N}$ such that $m_d(k,n') < (\overline{\mu_d}^{(s)}(k) + \gamma/4) \binom{n'-d}{k-d}$ for all $n' \in k\mathbb{N} + s$ with $n' \geq n_0$,
and 
we may assume that $n$ is sufficiently larger than $n_0$ so that $n^{1/k} / 8 \geq n_0$, which implies $p_\ell n/2 \geq n_0$.
Let $U$ be a $p_\ell$-random subset of $V(\cH)$. 
Let $\mathcal{E}$ be the event that $\s{U} = (1\pm \eps)p_\ell n$ and $\delta_{d}(\cH[U]) \geq \frac{\overline{\mu_d}^{(s)}(k) + \gamma/3}{(k-d)!} (p_\ell n)^{k-d}$.
We show that $\mathcal{E}$ occurs a.a.s.
Note that by a Chernoff bound, we have that
\[
\pr{\s{U} \neq (1 \pm \eps)p_\ell n} \leq 2 \exp\left(-\frac{\eps^2}{3}p_\ell n\right) \leq \exp(-\Omega(n^{1/k})).
\]
Note that for each $S \in \binom{V(\cH)}{d}$, we have $d_\cH(S) \geq (\overline{\mu_d}^{(s)}(k) + \gamma) \binom{n-d}{k-d} \geq \frac{\overline{\mu_d}^{(s)}(k) + \gamma/2}{(k-d)!} n^{k-d}.$
By \cref{lemma:prob_lemma} \ref{typical_i} and a union bound, with probability at least $1-\exp(-n^{1/(11k^2)})$, we have 
\[
\delta_d(\cH[U]) \geq \frac{\overline{\mu_d}^{(s)}(k) + \gamma/3}{(k-d)!} (p_\ell n)^{k-d}.
\]
Hence, $\mathcal{E}$ occurs a.a.s. We show that in this case $U$ is $(\cH, \eps)$-OM-stable for $\cH$. 
Let $M$ be a matching in $\cH$ such that $|V(M) \cap U| \leq \eps |U|$ and $V(\cH) \setminus V(M) \subseteq U$. Let $U' \coloneqq V(\cH) \setminus V(M)$.
Note that
\begin{equation*}
\delta_d(\cH[U']) \geq \frac{\overline{\mu_d}^{(s)}(k) + \gamma/3}{(k-d)!} (p_\ell n)^{k-d} - \eps |U|^{k-d} \geq \frac{\overline{\mu_d}^{(s)}(k) + \gamma/4}{(k-d)!} |U'|^{k-d} \geq \left(\overline{\mu_d}^{(s)}(k) + \frac{\gamma}{4}\right) \binom{|U'|-d}{k-d},
\end{equation*}
and
since $|U'| \geq p_\ell n / 2 \geq n_0$
and $|U'| \equiv n \modu{k}$, we have $\delta_d(\cH[U']) \geq (\overline{\mu_d}^{(s)}(k) + \gamma/4) \binom{|U'|-d}{k-d} \geq m_d(k,|U'|)$.  Therefore, it follows from the definition of $m_d(k, |U'|)$ that $\cH[U'] = \cH - V(M)$ contains an optimal matching, as desired.
\end{proof}

We are now ready to prove \cref{thm:main_general_fknp}.
\begin{proof}[Proof of Theorem~\ref{thm:main_general_fknp}]
Let $1/n \ll \delta \ll \eps_* \ll \eps \ll \gamma, 1/k$ with $\gamma \in (0,1)$ and $k \geq 3$. 
Let $0 \leq s \leq k-1$ be an integer such that $n \in k\mathbb{N} + s$.
Note that $\delta_d(\cH) \geq (\overline{\mu_d}^{(s)}(k) + \gamma) \binom{n-d}{k-d} \geq \frac{\overline{\mu_d}^{(s)}(k) + \gamma/2}{(k-d)!} n^{k-d}$. Moreover, we have
\begin{align*}
    \delta_1(\cH) \geq \frac{1}{\binom{k-1}{d-1}}\binom{n-1}{d-1}\delta_d(\cH) \geq \frac{\overline{\mu_d}^{(s)}(k) + \gamma/2}{2(k-1)!} n^{k-1}.
\end{align*}
Thus, since $\eps \ll \gamma$, $\cH$ is $\left(\frac{\overline{\mu_d}^{(s)}(k) + 4 \eps}{(k-d)!}, \frac{\overline{\mu_d}^{(s)}(k) + \gamma/2}{2(k-1)!}, d, \eps_*\right)$-dense. By \cref{lem:spread_OM,lem:d_deg_OM-stable}, there exists a probability measure on the set of optimal matchings of $\cH$ which is $\frac{1}{\delta n^{k-1}}$-spread, as desired.
\end{proof}



\subsection{\texorpdfstring{Proof of Theorem~\ref{thm:main_rrs_fknp}}{Proof of  Theorem~\ref{thm:main_rrs_fknp}}}

Now we briefly describe the following critical example mentioned in~\cite[Section 3]{Rodl2009}. Note that the critical example for odd $k$ was introduced in~\cite{KO2006}.

\begin{definition}[$\cH^0(k,n)$]\label{def:critical}
Let $k,n \geq 2$ be positive integers such that $n$ is divisible by $k$.
Let $\cH^0(k,n)$ be a $k$-uniform $n$-vertex hypergraph with an ordered partition $(A,B)$ of $V(\cH^0(k,n))$ such that the following holds.
\begin{itemize}
    \item If $k$ is odd, then $\s{A}$ is the unique odd integer in 
        $\{ \frac{n}{2} -1, \frac{n}{2} - \frac{1}{2}, \frac{n}{2}, \frac{n}{2} + \frac{1}{2} \}$
        and $E(\cH^0(k,n))$ is the collection of all subsets of size $k$ in $V(\cH^0(k,n)) = A \cup B$ which intersect $A$ in an even number of vertices.
    \item Otherwise if $k$ is even, then
        \[
        \s{A} = \begin{cases} \frac{n}{2}, &\text{if $\frac{n}{k}$ is odd and $\frac{n}{2}$ is even,} \\ \frac{n}{2} -1, &\text{ otherwise (thus $n/k \equiv n/2 \modu{2})$,} \end{cases}
        \]
        and $E(\cH^0(k,n))$ is the collection of all subsets of size $k$ in $V(\cH^0(k,n)) = A \cup B$ which intersect $A$ in an odd number of vertices.
\end{itemize}
\end{definition}

Let $\delta^0 (k,n) \coloneqq \delta_{k-1}(\cH^0(k,n))$. If $k$ is odd, then
\begin{align*}
    \delta^0(k,n) =  
    \begin{cases}
        n/2 + 1 - k & \text{for } n \equiv 0,2\:({\rm mod}\:4)\\
        n/2 + 1/2 - k & \text{for } n \equiv 1\:({\rm mod}\:4)\\
        n/2 + 3/2 - k & \text{for } n \equiv 3\:({\rm mod}\:4).
    \end{cases}
\end{align*}

Otherwise if $k$ is even, then
\begin{align*}
    \delta^0(k,n) = 
    \begin{cases}
        n/2 + 1 - k & \text{if $n/k$ is even}\\
        n/2 + 1 - k & \text{if $n/k$ is odd and $k/2$ is odd}\\
        n/2 + 2 - k & \text{if $n/k$ is odd and $k/2$ is even.}
    \end{cases}
\end{align*}

Note that $\cH^0(k,n)$ does not contain a perfect matching (for example, see~\cite[Section 3]{Rodl2009}), so $m_{k-1}(k,n) \geq \delta^0(k,n) + 1$ if $k \mid n$. In fact, R\"{o}dl, Ruci\'nski, and Szemer\'edi~\cite{Rodl2009} showed that $m_{k-1}(k,n) =  \delta^0(k,n) + 1$ when $k \geq 3$, $k \mid n$, and $n$ is sufficiently large.

We may also use the following definition from~\cite[Definition 3.3]{Rodl2009}. 
\begin{definition}[$\eps$-containment]\label{def:epscontain}
For any $\eps \in (0,1)$, an $n$-vertex $k$-uniform hypergraph $\cH$ \emph{$\eps$-contains} another $n$-vertex $k$-uniform hypergraph $\cG$ (or $\cG \subseteq_{\eps} \cH$) if there exists an isomorphic copy $\cH'$ of $\cH$ such that $V(\cH') = V(\cG)$ and $|\cG \setminus \cH'| \leq \eps n^k$.
\end{definition}

In the proof of Theorem~\ref{thm:main_rrs_fknp}, we must consider two cases according to whether $\cH$ is close to being critical.
The following two lemmas give that \aas a small random subset of vertices is OM-stable in both cases.

\begin{lemma}\label{lem:non-extremal_is_OM_stable}
Let $1/n \ll \eps \ll 1/k \leq 1/3$ such that $k \mid n$.
Let $\cH$ be a $k$-uniform $n$-vertex hypergraph with $\delta_{k-1}(\cH) \geq (1/2 - 1/\log n)n$ such that $\cH$ $\eps$-contains neither $\cH^0(k,n)$ nor $\overline{\cH^0(k,n)}$. Let $\ell \coloneqq \left\lceil\frac{k-1}{k} \log_2(n)\right\rceil$, $C_\ell \coloneqq \sum_{i=1}^\ell 2^{-i}$, and $p_\ell \coloneqq \frac{1}{C_\ell 2^{\ell}}$. Let $U$ be a $p_\ell$-random subset of $V(\cH)$. Then \aas $U$ is $(\cH, \eps)$-OM-stable for $\cH$.
\end{lemma}

\begin{lemma}\label{lem:extremal_OM}
Let $1/n \ll \eps \ll \eta \ll 1/k \leq 1/3$ such that $k \mid n$.
Let $\cH$ be a $k$-uniform $n$-vertex hypergraph $\cH$ with $\delta_{k-1}(\cH) \geq m_{k-1}(k,n) = \delta^0(k,n) + 1$ such that $\cH$ $\eps$-contains either $\cH^0(k,n)$ or $\overline{\cH^0(k,n)}$.
Let $\ell \coloneqq \lceil \frac{k-1}{k} \log_2 n \rceil$, $C_\ell \coloneqq \sum_{i=1}^{\ell} 2^{-i}$, and $p_\ell \coloneqq 1/(C_\ell 2^\ell)$.
There are at least $\eps n^{k-1}$ choices of an edge $e^* \in \cH$ such that for each of the choices of $e^*$, there exists a spanning subhypergraph $\cH'$ of $\cH - V(e^*)$ such that 
\begin{enumerate}[label = {$({\rm O}\arabic*)$}, leftmargin=\widthof{O100000}]
    \item\label{O1} $\cH'$ is $(1/2 - \eta ,\: \frac{0.15}{3^{k-1}(k-1)!}, \: k - 1 ,\: \eta)$-dense, and
    
    \item\label{O2} \aas a $p_\ell$-random subset of $V(\cH) - V(e^*)$ is $(\cH', \eta)$-OM-stable for $\cH - V(e^*)$.
\end{enumerate}
\end{lemma}

We will prove both lemmas in the next two sections. Subject to these lemmas, we prove Theorem~\ref{thm:main_rrs_fknp}.

\begin{proof}[Proof of Theorem~\ref{thm:main_rrs_fknp}]
Let $1/n_0 \ll \delta \ll \eps \ll \eta \ll 1/k \leq 1/3$. 
If $\cH$ $\eps$-contains neither $\cH^0(k,n)$ nor $\overline{\cH^0(k,n)}$, then Theorem~\ref{thm:main_rrs_fknp} follows by Lemmas~\ref{lem:spread_OM} and~\ref{lem:non-extremal_is_OM_stable}, since $m_{k-1}(n) = 1/k$, $\delta_{k-1}(\cH) \geq m_{k-1}(k,n) \geq n/2 - O(k)$, and $\delta_1(\cH) \geq \frac{1}{k-1} \binom{n-1}{k-2} \delta_{k-1}(\cH) \geq \frac{n^{k-1}}{3 (k-1)!}$.
Thus, we may assume that $\cH$ $\eps$-contains either $\cH^0(k,n)$ or $\overline{\cH^0(k,n)}$.

By Lemma~\ref{lem:extremal_OM}, there are at least $\eps n^{k-1}$ choices of an edge $e^* \in \cH$ satisfying~\ref{O1} and~\ref{O2}. 
We choose one of them uniformly at random and let $M^* \coloneqq \{e^* \}$. For each of the choices of $e^*$, there exists a spanning subhypergraph $\cH'$ of $\cH - V(e^*)$ which is $(1/2 - \eta ,\: \frac{0.15}{3^{k-1}(k-1)!} , \: k - 1, \: \eta)$-dense by~\ref{O1}, so $\cH'$ is $(1/k + 3\eta,\: \frac{0.15}{3^{k-1}(k-1)!},\: k - 1, \: \eps)$-dense.
By~\ref{O2} and Lemma~\ref{lem:spread_OM}, there exists a probability measure $\nu$ on the set of perfect matchings $M'$ of $\cH - V(e^*)$ that is $\frac{1}{\delta n^{k-1}}$-spread, conditioning on the choice of $e^*$. Let $M'$ be chosen randomly according to $\nu$, and let $M \coloneqq M^* \cup M'$. For any disjoint $e_1 , \dots , e_t \in \cH$,
\begin{align*}
    \pr{e_1 , \dots , e_t \in M} &\leq \pr{e_1, \dots, e_t \in M' \: | \: M^*} +
    \sum_{i=1}^{t} \pr{M^* = \{ e_i \}} \pr{ \{e_1 , \dots , e_t \}  \setminus \{e_i \} \subseteq M' \: | \: M^*}\\
    &\leq \left ( \frac{1}{\delta n^{k-1}} \right )^{t} + t \cdot \frac{1}{\delta n^{k-1}} \cdot \left ( \frac{1}{\delta n^{k-1}} \right )^{t-1} \leq \left ( \frac{e}{\delta n^{k-1}} \right )^{t}.
\end{align*}
Thus, the distribution of $M$ is $\frac{e}{\delta n^{k-1}}$-spread, as desired.
\end{proof}

\section{Proof of Lemma~\ref{lem:non-extremal_is_OM_stable}}\label{sec:nonextremal}

Roughly speaking the proof of \cref{lem:non-extremal_is_OM_stable} proceeds as follows. We show that $\cH$ contains many small absorbing structures. We then use \cref{lemma:prob_lemma} to show that a $p_\ell$-random subset of vertices~$U$ still contains many of these small absorbers. We use these to build a larger absorbing matching~$M$ of size $O(\log^4(n))$ in $\cH[U]$. The vertices of~$M$ will be the set~$U_*$ of protected vertices that is allowed by the definition of $(\cH, \eps)$-OM-stable. We let~$\widetilde{M}$ be any matching in $\cH$ such that $U_* \subseteq V(\cH) \setminus V(\widetilde{M}) \subseteq U$ and $|V(\widetilde{M}) \cap U| \leq \eps \s{U}$. Then the minimum codegree of $\cH - V(\widetilde{M}) - V(M)$ is still large enough to guarantee a matching that either covers all vertices or all but exactly $k$ vertices.
Finally, we use the absorbing property of~$M$ to complete this matching to a perfect matching in $V(\cH) - V(\widetilde{M})$.

Now we define the absorbing structures that were introduced in~\cite[Definitions 5.1 and 5.2]{Rodl2009}.

\begin{definition}[$S$-absorbing $k$-matchings and $S$-absorbing $(k+1)$-matchings]
Let $\cH$ be a $k$-uniform hypergraph and $S = \{x_1, \dots, x_k\} \in \binom{V(\cH)}{k}$. 

A $k$-matching $\{e_1, \dots, e_k\}$ in $\cH$ is \emph{$S$-absorbing} if there exists a $(k+1)$-matching $\{e'_1, \dots, e'_k, f\}$ in $\cH$ such that 
\begin{enumerate}[label = {$(\text{AM}\arabic*)$}, leftmargin= \widthof{AM100000}]
    \item $e_i \cap e'_j = \varnothing$ for all $i \neq j$,
    \item $e'_i \setminus e_i = \{x_i\}$ and $\{y_i\} \coloneqq e_i \setminus e'_i$ for all $i \in [k]$, and
    \item $f = \{y_1, \dots, y_k\}$.
\end{enumerate}

A $(k+1)$-matching $\{e_0, \dots, e_k\}$ in $\cH$ is \emph{$S$-absorbing} if there exists a $(k+2)$-matching $\set{e'_1, {\dots, e'_k}, f, f'}$ in $\cH$ such that
\begin{enumerate}[label = {$(\text{AM}\arabic*')$}, leftmargin= \widthof{AM'100000}]
    \item $e_i \cap e'_j = \varnothing$ for all $i \neq j$,
    \item $e'_i \setminus e_i = \{x_i\}$ and $\{y_i\} \coloneqq e_i \setminus e'_i$ for all $i \in [k]$, and
    \item $f \cap e_1 = \{y_1\} = f \setminus e_0$, $f' = \{y_0, y_2, \dots, y_k\}$, where $\{y_0\} \coloneqq e_0 \setminus f$.
\end{enumerate}
\end{definition}


The next lemma follows from \cite[Claim 5.1]{Rodl2009} and \cite[Fact 5.3]{Rodl2009} (see Definition~\ref{def:critical} for the definition of $\cH^0(k,n)$). It shows that in the setting of \cref{lem:non-extremal_is_OM_stable}, $\cH$ has many $S$-absorbing matchings for each set~$S$ of~$k$ vertices.

\begin{lemma}[\cite{Rodl2009}]\label{lem:no_of_matchings}
Let $1/n \ll \eps , 1/k \leq 1/3$ such that $k \mid n$. Let $\cH$ be a $k$-uniform hypergraph on $n$ vertices with $\delta_{k-1}(\cH) \geq (1/2 - 1/\log n)n$ such that $\cH^0(k,n) \not\subseteq_\eps \cH$ and $\overline{\cH^0(k,n)} \not\subseteq_\eps \cH$. Then at least one of the following holds. 
\begin{enumerate}[label = {\upshape{(\alph*})}, leftmargin= \widthof{a00000}]
    \item For every $S \subseteq V(\cH)$ with $\s{S} = k$, there are $\Omega(n^{k^2}/\log^3(n))$ many $S$-absorbing $k$-matchings in $\cH$. \label{SAMa}
    \item For every $S \subseteq V(\cH)$ with $\s{S} = k$, there are $\Omega(n^{k^2+k}/ \log^3(n))$ many $S$-absorbing $(k+1)$-matchings in $\cH$. \label{SAMb}
\end{enumerate}
\end{lemma}

The next lemma follows from the proof of \cite[Fact 5.4]{Rodl2009}. It says that if we have many $S$-absorbing matchings for each set $S$ of $k$ vertices in $\cH$ then we can build an absorbing matching of size $O(\log^4(n))$ that can absorb any set of $k$ vertices.

\begin{lemma}[\cite{Rodl2009}] \label{lem:absorbing_matching}
Let $1/n \ll 1/k \leq 1/3$. Let $\cH$ be a $k$-uniform $n$-vertex hypergraph. Suppose that at least one of the following holds.
\begin{enumerate}[label = {\upshape{(\alph*})}, leftmargin= \widthof{a00000}]
    \item For every $S \subseteq V(\cH)$ with $\s{S} = k$, there are $\Omega(n^{k^2}/\log^3(n))$ many $S$-absorbing $k$-matchings in $\cH$. \label{SAMa'}
    \item For every $S \subseteq V(\cH)$ with $\s{S} = k$, there are $\Omega(n^{k^2+k}/ \log^3(n))$ many $S$-absorbing $(k+1)$-matchings in $\cH$. \label{SAMb'}
\end{enumerate}
Then $\cH$ contains a matching $M$ of size $O(\log^4(n))$ such that for each set $S \subseteq V(\cH) \setminus V(M)$ with $\s{S} =k$, there exists a perfect matching in $\cH[V(M) \cup S]$.
\end{lemma}


The following corollary is a direct application of  Lemma~\ref{lemma:prob_lemma}. We use it to show that for a $p_\ell$-random subset $U$ of vertices of $\cH$, the property of $\cH$ of having many $S$-absorbing matchings is inherited \aas by $\cH[U]$.

\begin{corollary} \label{cor:M_count_in_U}
Let $1/n \ll 1/s \leq 1/k \leq 1/3$. Let $\cH$ be a $k$-uniform $n$-vertex hypergraph. Let $\ell \coloneqq \left\lceil\frac{k-1}{k} \log_2(n)\right\rceil$, $C_\ell \coloneqq \sum_{i=1}^\ell 2^{-i}$, and $p_\ell \coloneqq \frac{1}{C_\ell 2^{\ell}}$. Let $U$ be a $p_\ell$-random subset of $V(\cH)$, and let $\cM$ be a set of $s$-matchings in $\cH$ with $\s{\cM} = \Omega(n^{sk}/ \log^3(n))$. Then with probability at least $1 -  \exp(-n^{1/6sk^2})$, the number of matchings in $\cM$ that are contained in $\cH[U]$ is $\Omega((p_\ell n)^{sk} / \log^3(p_\ell n))$.
\end{corollary}

To prove \cref{lem:non-extremal_is_OM_stable}, we also need the following result by Han~\cite{Han2015}.

\begin{theorem}[{\cite[Theorem 1.1]{Han2015}}] \label{thm:Han_OM}
Let $1/n \ll 1/k \leq 1/3$ such that $k$ does not divide $n$. Let $\cH$ be a $k$-uniform hypergraph on $n$ vertices with $\delta_{k-1}(\cH) \geq \lfloor n/k \rfloor$. Then $\cH$ contains an optimal matching.
\end{theorem}

Now we are ready to prove Lemma~\ref{lem:non-extremal_is_OM_stable}.

\begin{proof}[Proof of Lemma~\ref{lem:non-extremal_is_OM_stable}]
By \cref{lem:no_of_matchings}, at least one of the following holds. 
\begin{enumerate}[label = {\upshape{(\alph*})}, leftmargin= \widthof{a00000}]
    \item For every $S \subseteq V(\cH)$ with $\s{S} = k$, there are $\Omega(n^{k^2}/\log^3(n))$ many $S$-absorbing $k$-matchings in $\cH$. \label{SAMa''}
    \item For every $S \subseteq V(\cH)$ with $\s{S} = k$, there are $\Omega(n^{k^2+k}/ \log^3(n))$ many $S$-absorbing $(k+1)$-matchings in $\cH$. \label{SAMb''}
\end{enumerate}
Suppose that \labelcref{SAMa''} holds (the proof for if \labelcref{SAMb''} holds is similar). Let $n_* \coloneqq \s{U}$. We have that \aas $n_* = (1\pm \eps) p_\ell n$ and $\delta_{k-1}(\cH[U]) \geq (1/2 - 2\eps)\s{U}$. By \cref{cor:M_count_in_U} and a union bound, it follows that \aas for every $S \in \binom{U}{k}$, the number of $S$-absorbing $k$-matchings in $\cH[U]$ is $\Omega(n_*^{k^2}/ \log^3(n_*))$. Suppose that all of these events occur. By \cref{lem:absorbing_matching}, there exists a matching $M$ in $\cH[U]$ of size $O(\log^4 (n_*))$ such that for each set $S \subseteq U \setminus V(M)$ with $\s{S} = k$, there exists a perfect matching in $\cH[V(M) \cup S]$. Let $U_* \coloneqq V(M)$, and note that $\s{U_*} \leq \eps \s{U}$. 
Let $\widetilde{M}$ be a matching in $\cH$ such that $U_* \subseteq V(\cH) \setminus V(\widetilde{M}) \subseteq U$ and $|V(\widetilde{M}) \cap U| \leq \eps \s{U}$.
Let $U' \coloneqq V(\cH) \setminus V(\widetilde{M})$, and note that $\s{U'} \equiv n \equiv 0 \modu{k}$. Let $u \in U' \setminus U_*$ and $U'' \coloneqq U' \setminus (U_* \cup \{u\})$. Note that $\s{U''} \equiv k-1 \modu{k}$ and $\delta_{k-1}(\cH[U'']) \geq \s{U''}/k$. Thus, by \cref{thm:Han_OM}, $\cH[U'']$ contains a matching $M_*$ covering all but a set $S_*$ of $k-1$ vertices of $U''$. Let $S \coloneqq S_* \cup \{u\}$. By the absorption property of $M$, $\cH[U_* \cup S]$ contains a perfect matching $M'$. Note that $M_* \cup M'$ is a perfect matching in $\cH - V(\widetilde{M}) = \cH[U']$. Hence, $U$ is $(\cH, \eps)$-OM-stable for $\cH$. 
\end{proof}



\section{Proof of Lemma~\ref{lem:extremal_OM}}\label{sec:extremal}

Now we briefly sketch the proof of Lemma~\ref{lem:extremal_OM}. 
Since $\cH$ is close to being a critical hypergraph, there are $\Omega(n^{k-1})$ many `atypical' edges (see Lemma~\ref{lem:manyatypical}). After choosing one of them (say $e^*$) and deleting the vertices from $V(e^*)$, the resulting hypergraph $\cH - V(e^*)$ will meet the `divisibility condition' (see Definition~\ref{def:div}) which ensures a perfect matching even though the minimum codegree is slightly below $m_{k-1}(k,n)$ (see Theorem~\ref{thm:perfectmatching_extremal}).
For a spanning subhypergraph $\cH'$ of $\cH - V(e^*)$ which consists of all typical edges of $\cH - V(e^*)$, by the definition of typical edges, $\cH'$ is also close to being a critical hypergraph. Thus, $\cH'$ is `dense' enough to satisfy~\ref{O1}.
To show~\ref{O2}, for a $p_\ell$-random subset $U_\ell$ of $V(\cH')$, we need to make sure that $\cH - V(e^*) - V(M')$ has a perfect matching for any matching $M'$ of $\cH'$ with $2 \mid |M'|$ and $|V(M') \cap U_\ell| = o(|U_\ell|)$. Using the structural properties of $\cH'$, we can show that $\cH - V(e^*) - V(M')$ is close to being a critical hypergraph and also meets the divisibility condition. Thus, by Theorem~\ref{thm:perfectmatching_extremal}, $\cH - V(e^*) - V(M')$ has a perfect matching.

Since the proof of Lemma~\ref{lem:extremal_OM} relies on some structural information of $\cH$, we need to introduce several notations first.

Let $k \geq 3$ be an integer, let $0 \leq r \leq k$ be an integer, and let $A$ and $B$ be disjoint sets.
Let $\cK_r(A,B) \coloneqq \{ e \subseteq A \cup B \colon \:|e|=k,\:|e \cap A| = r,\:|e \cap B| = k-r \}$. 
For any $k$-uniform hypergraph $\cH$ with $V(\cH) = A \cup B$, let $E_\cH^j(A,B) \coloneqq \cH \cap \cK_j(A,B) = \{ e \in \cH : |e \cap A| = j \}$. We often omit the subscript $\cH$ if it is clear.
Extending the definition of $\cH^0(k,n)$, let us define
\begin{align*}
    \cH^0(k,A,B) \coloneqq
    \begin{cases}
        \bigcup_{\text{$r$: even}}\cK_r(A,B) & \text{for odd $k$}\\
        \bigcup_{\text{$r$: odd}}\cK_r(A,B) & \text{for even $k$.}
    \end{cases}
\end{align*}

Note that
\begin{align*}
    \overline{\cH^0(k,A,B)} =
    \begin{cases}
        \bigcup_{\text{$r$: odd}}\cK_r(A,B) = \bigcup_{\text{$r$: even}}\cK_r(B,A) & \text{for odd $k$}\\
        \bigcup_{\text{$r$: even}}\cK_r(A,B) & \text{for even $k$.}
    \end{cases}
\end{align*}


Let $n \in \mathbb{N}$ divisible by $k$. If $k$ is odd, then let $a(k,n)$ be the unique odd integer from $\{(n + \ell)/2 : \ell \in \mathbb{Z},\:-2 \leq \ell \leq 1 \}$. Otherwise if $k$ is even, then let
\begin{align*}
    a(k,n) \coloneqq
    \begin{cases}
        n/2 - 1 & \text{for even $n/k$}\\
        n/2 - 1 & \text{for odd $n/k$ and odd $n/2$}\\
        n/2 & \text{for odd $n/k$ and even $n/2$.}
    \end{cases}
\end{align*}

\begin{definition}[Standard ordered pair]
Let $k,n \geq 3$ be positive integers such that $k \mid n$.
Let $A$ and $B$ be disjoint sets such that $|A|+|B|=n$. An ordered pair $(A,B)$ is \emph{standard} if $|A| = a(k,n)$ and $|B| = n - a(k,n)$.
\end{definition}

Note that $\cH^0(k,n)$ is a $k$-uniform $n$-vertex hypergraph isomorphic to $\cH^0(k,A,B)$ for a standard ordered pair $(A,B)$.

\begin{definition}[Types]\label{def:type}
Let $k,n \geq 3$ be positive integers such that $k \mid n$, and let $\cH$ be a $k$-uniform $n$-vertex hypergraph.  For $\eps \in (0,1)$ and an ordered partition $(A,B)$ of $V(\cH)$ such that either  
$|\cH^0(k,A,B) \setminus \cH| \leq \eps n^k$ or $|\overline{\cH^0(k,A,B)} \setminus \cH| \leq \eps n^k$ holds, we define the following.

\begin{enumerate}[{\rm (\alph*)}]
    \item If $k$ is odd and $|\cH^0(k,A,B) \setminus \cH| \leq \eps n^k$, then we say $\cH$ \emph{belongs to the type $(\rm a)$} with respect to $(\eps,A,B)$. 
    
    \item If $k$ is odd and $|\overline{\cH^0(k,A,B)} \setminus \cH| \leq \eps n^k$, then we say $\cH$ \emph{belongs to the type $(\rm b)$} with respect to $(\eps,A,B)$. 
    
    \item If $k$ is even and $|\overline{\cH^0(k,A,B)} \setminus \cH| \leq \eps n^k$, then we say $\cH$ \emph{belongs to the type $(\rm c)$} with respect to $(\eps,A,B)$. 
    
    \item If $k \equiv 0\:({\rm mod}\:4)$ and $|\cH^0(k,A,B) \setminus \cH| \leq \eps n^k$, then we say $\cH$ \emph{belongs to the type $(\rm d)$} with respect to $(\eps,A,B)$. 
    
    \item If $k \equiv 2\:({\rm mod}\:4)$ and $|\cH^0(k,A,B) \setminus \cH| \leq \eps n^k$, then we say $\cH$ \emph{belongs to the type $(\rm e)$} with respect to $(\eps,A,B)$. 
\end{enumerate}
We also say $\cH$ \emph{belongs to the type $\alpha$} if it belongs to the type $\alpha$ with respect to $(\eps, A, B)$ for some $\eps \in (0, 1)$ and partition $(A, B)$ of $V(\cH)$.
\end{definition}

\begin{definition}[Typical indices and edges]\label{def:typicalindex}
Let $k \geq 3$ be a positive integer.
For $\alpha \in \{ {\rm (a)} , {\rm (b)} , {\rm (c)} ,\allowbreak {\rm (d)} ,\allowbreak {\rm (e)} \}$, an index $r \in \{0 , \dots , k\}$ is called $\alpha$-\emph{typical} (with respect to $k$) if
\begin{align*}
    r \equiv
    \begin{cases}
        0\:({\rm mod}\:2) & \text{for $\alpha \in \{ {\rm (a)},{\rm (c)} \}$}\\
        1\:({\rm mod}\:2) & \text{for $\alpha \in \{ {\rm (b)},{\rm (d)},{\rm (e)} \}$}.
    \end{cases}
\end{align*}
Otherwise $r$ is called $\alpha$-\emph{atypical}.

For any $k$-uniform hypergraph $\cH$ with an ordered partition $(A,B)$, an edge $e \in \cH$ is \emph{$\alpha$-typical} with respect to $(A,B)$ if $e \in E_\cH^r(A,B)$ for an $\alpha$-typical index $r$. Otherwise an edge $e$ is called \emph{$\alpha$-atypical} with respect to $(A,B)$.
\end{definition}

\begin{obs}\label{obs:epscontain_type}
%
%
Let $k \geq 3$ be a positive integer, and let $\alpha \in \{ {\rm (a)} , {\rm (b)} , {\rm (c)} , {\rm (d)} ,\allowbreak {\rm (e)} \}$.  For disjoint sets $A$ and $B$,
\begin{equation*}
    \bigcup_{r:\alpha\text{-typical}} \cK_r(A,B) = \left\{\begin{array}{l l}
    \cH^0(k,A,B) & \text{ if } \alpha \in \{ {\rm (a)} , {\rm (d)} , {\rm (e)} \} \text{ and}\\
    \overline{\cH^0(k,A,B)} & \text{ if } \alpha \in \{ {\rm (b)} , {\rm (c)} \}.
    \end{array}\right.
\end{equation*}
In particular, for $\eps \in (0, 1)$, a $k$-uniform hypergraph $\cH$ belongs to the type $\alpha$ with respect to $(\eps, A, B)$ if and only if
\begin{equation*}
    \sum_{r:\alpha\text{-typical}} |\cK_r(A,B) \setminus E_\cH^r(A,B)| \leq \eps n^k.
\end{equation*}
\end{obs}

\begin{definition}[Special typical index]\label{def:special}
Let $k \geq 3$ be a positive integer, and let $\cH$ be a $k$-uniform hypergraph which belongs to the type $\alpha \in \{ {\rm (a)} , {\rm (b)} , {\rm (c)} , {\rm (d)} , {\rm (e)} \}$. 
The \emph{special $\alpha$-typical index} for $\cH$ is $r^* \coloneqq k-1,\:1,\:k-2,\:k/2+1,\:k/2$ for $\alpha = {\rm (a)},\:{\rm (b)},\:{\rm (c)},\:{\rm (d)},\:{\rm (e)}$ respectively.
\end{definition}

\begin{proposition}\label{prop:specialtype}
Let $k \geq 3$ be a positive integer, and let $\cH$ be a $k$-uniform hypergraph which belongs to the type $\alpha \in \{ {\rm (a)} , {\rm (b)} , {\rm (c)} , {\rm (d)} , {\rm (e)} \}$.
Then the special $\alpha$-typical index for $\cH$ is $\alpha$-typical.
\end{proposition}
\begin{proof}
By Definitions~\ref{def:type} and~\ref{def:special}, since $\cH$ belongs to the type $\alpha$, the following hold.
\begin{itemize}
    \item If $\alpha = {\rm (a)}$, then $k$ is odd, so the special index $k-1$ is even. 
    
    \item If $\alpha = {\rm (b)}$, then the special index is $1$ which is odd. 
    
    \item If $\alpha = {\rm (c)}$, then $k$ is even, so the special index $k-2$ is even. 
    
    \item If $\alpha = {\rm (d)}$, then $4 \mid k$, so the special index $k/2 + 1$ is odd.
    
    \item If $\alpha = {\rm (e)}$, then $k \equiv 2 \modu{2}$, so the special index $k/2$ is odd. 
\end{itemize}
Thus, by Definition~\ref{def:typicalindex}, the special $\alpha$-typical index for $\cH$ is $\alpha$-typical.
\end{proof}

\begin{definition}[Divisibility condition]\label{def:div}
Let $n$ be a positive integer divisible by $k$.
Let $(A,B)$ be an ordered pair such that $n = |A|+|B|$. We say $(A,B)$ satisfies the \emph{divisibility condition} with respect to the type $\alpha \in \{ {\rm (a)} , {\rm (b)} , {\rm (c)} , {\rm (d)} , {\rm (e)} \}$ if the following hold.
\begin{itemize}
    \item If $\alpha \in \{ {\rm (a)} , {\rm (c)} \}$, then $|A|$ is even.
    
    \item If $\alpha = {\rm (b)}$, then $|B|$ is even.
    
    \item If $\alpha = {\rm (d)}$, then $\frac{|A| - |B|}{2} \equiv \frac{n}{k}\: ({\rm mod}\:2)$. 
    
    \item If $\alpha = {\rm (e)}$, then $\frac{|A| - |B|}{2}$ is even.
\end{itemize}
\end{definition}

Now we state two ingredients from~\cite{Rodl2009} which we use in the proof of Lemma~\ref{lem:extremal_OM}.
Here we briefly explain how to deduce the following theorem from the proof of~\cite[Lemma 3.1]{Rodl2009}: the hypergraph $\cH$ in~\cite[Lemma 3.1]{Rodl2009} is only assumed to satisfy $\delta_{k-1}(\cH) \geq \delta^0(k,n) + 1$ and that $\cH$ $\eps$-contains either $\cH^0(k,n)$ or $\overline{\cH^0(k,n)}$.
In their proof, they began with slightly modifying the standard ordered partition $(A,B)$ to $(A',B')$ to ensure that  $d_{E_\cH^{r^*}(A',B')}(v) > 0.1 d_{K_{r^*}(A',B')}(v)$ for each $v \in V(\cH)$ and the special $\alpha$-typical index $r^*$ for $\cH$. 
Then they used Facts 4.5--4.8 of~\cite{Rodl2009} which provides an atypical edge $e$, and showed that the partition $(A' \setminus V(e) , B' \setminus V(e))$ satisfies the divisibility condition if $(A',B')$ does not satisfy the divisibility condition.
Since the rest of their proof works for the hypergraphs with minimum codegree at least $n/2 - o(n)$, the minimum degree condition can be relaxed to $\delta_{k-1}(\cH) \geq n/2 - o(n)$ if we further assume that $(A',B')$ satisfies the divisibility condition and that  $d_{E_\cH^{r^*}(A',B')}(v) > 0.1 d_{K_{r^*}(A',B')}(v)$ for each $v \in V(\cH)$, as we stated as below.

\begin{theorem}[\cite{Rodl2009}]\label{thm:perfectmatching_extremal}
Let $1/n \ll \eps \ll 1/k \leq 1/3$ such that $k \mid n$.
Let $A'$ and $B'$ be disjoint sets such that $n = |A'|+|B'|$ and $||A'| - |B'|| \leq \eps n$.
If $\cH$ is a $k$-uniform $n$-vertex hypergraph with an ordered partition $(A',B')$ of $V(\cH)$, then $\cH$ has a perfect matching if the following hold.
\begin{enumerate}[{\rm (i)}]
    \item $\delta_{k-1}(\cH) \geq n/2 - \eps n$.
    
    \item The hypergraph $\cH$ belongs to some type $\alpha \in \{ {\rm (a)} , {\rm (b)} , {\rm (c)} , {\rm (d)} , {\rm (e)} \}$ with respect to $(\eps,A',B')$.
    
    \item\label{cond:div} The ordered partition $(A',B')$ satisfies the divisibility condition with respect to the type $\alpha$.
    
    \item For each vertex $v \in V(\cH)$, $d_{E_\cH^{r^*}(A',B')}(v) > 0.1 d_{\cK_{r^*}(A',B')}(v)$, where $r^*$ is the special $\alpha$-typical index for $\cH$.
\end{enumerate}
\end{theorem}


The following lemma shows that there are $\Omega(n^{k-1})$ atypical edges, which follows from the proofs of Facts 4.5--4.8 of~\cite{Rodl2009}.

\begin{lemma}[\cite{Rodl2009}]\label{lem:manyatypical}
Let $1/n \ll c \ll 1/k \leq 1/3$ such that $k \mid n$.
Let $\cH$ be a $k$-uniform $n$-vertex hypergraph such that $\delta_{k-1}(\cH) \geq \delta^0 (k,n) + 1$.
For any partition $\{A',B'\}$ of $V(\cH)$ such that $|A'|,|B'| \geq n/10$, the following hold.

\begin{enumerate}[{\rm (\alph*)}]
\item If $k$ is odd and $|A'|$ is odd, then $|E_\cH^1(A',B') \cup E_\cH^{k-2}(A',B')| \geq c n^{k-1}$.

\item If $k$ is odd and $|B'|$ is odd, then $|E_\cH^{k-1}(A',B') \cup E_\cH^{2}(A',B')| \geq cn^{k-1}$.

\item If $k$ is even, then $|E_\cH^1(A',B') \cup E_\cH^{k-1}(A',B')| \geq cn^{k-1}$.

\item If $k \equiv 0 \modu{4}$ and $\frac{|A'| - |B'|}{2} \not\equiv \frac{n}{k}\: ({\rm mod}\:2)$, then $|E_\cH^2(A',B') \cup E_\cH^{k-2}(A',B')| \geq cn^{k-1}$.

\item If $k \equiv 2 \modu{4}$, then $|E_\cH^2(A',B') \cup E_\cH^{k-2}(A',B')| \geq cn^{k-1}$.
\end{enumerate}
\end{lemma}

Now we are ready to prove Lemma~\ref{lem:extremal_OM}.

\begin{proof}[Proof of Lemma~\ref{lem:extremal_OM}]
Let $1/n \ll \delta \ll \eps \ll \eta \ll 1/k \leq 1/3$.
Since $\cH$ $\eps$-contains either $\cH^0(k, n)$ or $\overline{\cH^0(k, n)}$, there exists a standard partition $( A, B )$ of $V(\cH)$ such that $\cH$ belongs to the type $\alpha$ with respect to $(\eps,A,B)$ for some $\alpha \in \{ {\rm (a)},{\rm (b)},{\rm (c)},{\rm (d)},{\rm (e)} \}$. Let $r^*$ be the special $\alpha$-typical index for $\cH$.
By~\cite[Fact 4.4]{Rodl2009}, there exists an ordered partition $(A',B')$ of $V(\cH)$ such that the following hold.
\begin{enumerate}[label = {$(\text{S}\arabic*)$}, leftmargin= \widthof{S000000}]
    \item\label{S1} $|A \triangle A'| = |B \triangle B'| \leq \eps^{1/2} kn$, and thus $||A'|-|B'|| \leq 2 \eps^{1/2} kn$.\COMMENT{Since $A = V(\cH) \setminus B$ and $A' = V(\cH) \setminus B'$, $A \triangle A' = B \triangle B'$. One can check that $||A'|-|B'|| = ||A' \cap A| + |A' \cap B| - |B' \cap B| - |B' \cap A|| \leq ||A' \cap A| - |B' \cap B|| + |A' \cap B| + |B' \cap A| \leq ||A| - |B|| + |A \triangle A'| \leq 2 \eps^{1/2} kn$, since $|A' \cap B| + |B' \cap A| = |A' \setminus A| + |A \setminus A'| = |A \triangle A'|$.}
    
    \item\label{S2} For each vertex $v \in V(\cH)$, $d_{E_\cH^{r^*}(A',B')}(v) > 0.2 d_{\cK_{r^*}(A',B')}(v) > 0.2 \frac{n^{k-1}}{3^{k-1} (k-1)!}$.\COMMENT{Note that $d_{\cK_{r^*}(A',B')}(v) \geq \frac{\min(|A'|-k+1 , |B'|-k+1)^{k-1}}{(k-1)!} \geq \frac{(n/3)^{k-1}}{(k-1)!}.$}
\end{enumerate}

\begin{claim}\label{claim:epscontain}
$\cH$ belongs to the type $\alpha$ with respect to $(5k\eps^{1/2},A',B')$.
\end{claim}
\begin{claimproof}
Note that
\begin{align*}
    \sum |\cK_r(A',B') \setminus E_\cH^r(A',B')| &\leq \sum |\cK_r(A',B') \setminus \cK_r(A,B)|\\
    &+ \sum|\cK_r(A,B) \setminus E_\cH^r(A,B)|\\
    &+ \sum|E_\cH^r(A,B) \setminus E_\cH^r(A',B')|.
\end{align*}
where the summations are taken over all $\alpha$-typical indices $r$. 
By Observation~\ref{obs:epscontain_type}, since $\cH$ belongs to the type $\alpha$ with respect to $(\eps, A, B)$, 
the second term in this sum is at most $\eps n^k \leq k \eps^{1/2} n^k$.  By \ref{S1}, the first and third terms in this sum are each at most $2 \eps^{1/2} kn^k$.  Thus, again by Observation~\ref{obs:epscontain_type}, $\cH$ belongs to the type $\alpha$ with respect to $(5k\eps^{1/2},A',B')$, as desired.
\end{claimproof}

\begin{claim}\label{claim:div''}
There are at least $\eps n^{k-1}$ choices of an edge $e^* \in \cH$ such that for each of the choices of $e^*$, the subhypergraph $\cH - V(e^*)$ belongs to the type $\alpha$ with respect to $(6k\eps^{1/2},A'',B'')$, where $A'' \coloneqq A' \setminus V(e^*)$ and $B'' \coloneqq B' \setminus V(e^*)$, 
and the ordered partition $(A'',B'')$ satisfies the divisibility condition with respect to the type $\alpha$.
\end{claim}
\begin{claimproof}
    By Claim~\ref{claim:epscontain} and Observation~\ref{obs:epscontain_type}, $\cH - V(e)$ belongs to the type $\alpha$ with respect to $(6k \eps^{1/2}, A \setminus V(e), B\setminus V(e))$ for every $e \in \cH$, so it suffices to show that there are at least $\eps n^{k - 1}$ choices of an edge $e^*$ such that $(A' \setminus V(e^*), B' \setminus V(e^*))$ satisfies the divisibility condition with respect to $\alpha$. 
    
    First, if $(A', B')$ satisifies the divisibility condition for $\alpha$, then by the choice of the special typical index $r^*$, it is easy to see that the ordered partition $(A' \setminus V(e^*), B' \setminus V(e^*))$ satisfies the divisibility condition for every $e^* \in E_\cH^{r^*}(A',B')$.  In this case, \ref{S2} implies that there are sufficiently many choices for $e^*$.
    
    Thus, we may assume $(A', B')$ does not satisfy the divisibility condition.  Let
    \begin{equation*}
        E^* \coloneqq \left\{\begin{array}{l l}
        E_\cH^1(A',B') \cup E_\cH^{k-2}(A',B') & \text{ if $\alpha = {\rm (a)}$},\\
        E_\cH^{k-1}(A',B') \cup E_\cH^{2}(A',B') & \text{ if $\alpha = {\rm (b)}$},\\
        E_\cH^1(A',B') \cup E_\cH^{k-1}(A',B') & \text{ if $\alpha = {\rm (c)}$},\\
        E_\cH^2(A',B') \cup E_\cH^{k-2}(A',B') & \text{ if $\alpha \in \{{\rm (d), {\rm (e)}\}}$}.
        \end{array}\right.
    \end{equation*}
    Since $(A', B')$ does not satisfy the divisibility condition, it is also easy to see that in all cases of $\alpha$, the ordered partition $(A' \setminus V(e^*), B' \setminus V(e^*))$ satisfies the divisibility condition for every $e^* \in E^*$.  Moreover, by Lemma~\ref{lem:manyatypical}, we have $|E^*| \geq \eps n^{k - 1}$, so there are sufficiently many choices for $e^*$, as desired.
    \COMMENT{{\bf Case $\alpha = {\rm (a)}$}.
    In this case $k$ is odd and $r^* = k-1$ by Definitions~\ref{def:type} and~\ref{def:special}.
    If $|A'|$ is even, then by~\ref{S2}, there are at least $\eps n^{k-1}$ edges in $E_\cH^{r^*}(A',B')$ incident to a vertex $v \in V(\cH)$. We choose $e^* \in E_\cH^{r^*}(A',B')$. Since $r^* = k-1$ is even, $|A'| \equiv r^* = |e^* \cap A'|\:({\rm mod}\:2)$. 
    Otherwise if $|A'|$ is odd, by Lemma~\ref{lem:manyatypical},  $|E_\cH^1(A',B') \cup E_\cH^{k-2}(A',B')| \geq \eps n^{k-1}$, so we choose $e^* \in E_\cH^1(A',B') \cup E_\cH^{k-2}(A',B')$. Since $|e^* \cap A'| \in \{ 1 , k-2 \}$ and $k$ is odd, we have $|A'| \equiv 1 \equiv |e^* \cap A'|\:({\rm mod}\:2)$.
    Thus, for both cases, $|A''| = |A'| - |e^* \cap A'| \equiv 0\:({\rm mod}\:2)$, so $(A'',B'')$ satisfies the divisibility condition with respect to the type $\rm (a)$.
{\bf Case $\alpha = {\rm (b)}$}. 
    In this case $k$ is odd and $r^* = 1$ by Definitions~\ref{def:type} and~\ref{def:special}. If $|B'|$ is even, then by~\ref{S2}, there are at least $\eps n^{k-1}$ edges in $E_\cH^{r^*}(A',B')$ incident to a vertex $v \in V(\cH)$. We choose $e^* \in E_\cH^{r^*}(A',B')$. Note that $|B'| \equiv k-1 = k-r^* = |e^* \cap B'|\:({\rm mod}\:2)$ since $k$ is odd. 
    Otherwise if $|B'|$ is odd, by Lemma~\ref{lem:manyatypical},  $|E_\cH^{k-1}(A',B') \cup E_\cH^{2}(A',B')| \geq \eps n^{k-1}$, so we choose $e^* \in E_\cH^{k-1}(A',B') \cup E_\cH^{2}(A',B')$. Since $|e^* \cap B'| \in \{ 1, k-2 \}$ and $k$ is odd, 
    we have $|B'| \equiv 1 \equiv |e^* \cap B'|\:({\rm mod}\:2)$.
    Thus, for both cases, $|B''| = |B'| - |e^* \cap B'| \equiv 0\:({\rm mod}\:2)$, so $(A'',B'')$ satisfies the divisibility condition with respect to the type $\rm (b)$.
{\bf Case $\alpha = {\rm (c)}$}.
    In this case, $k$ is even and $r^* = k-2$ by Definitions~\ref{def:type} and~\ref{def:special}.
    If $|A'|$ is even, then by~\ref{S2}, there are at least $\eps n^{k-1}$ edges in $E_\cH^{r^*}(A',B')$ incident to a vertex $v \in V(\cH)$. We choose $e^* \in E_\cH^{r^*}(A',B')$. Since $r^* = k-2$ is even, $|A'| \equiv r^* = |e^* \cap A'|\:({\rm mod}\:2)$.
    Otherwise if $|A'|$ is odd, by Lemma~\ref{lem:manyatypical}, $|E_\cH^1(A',B') \cup E_\cH^{k-1}(A',B')| \geq \eps n^{k-1}$, so we choose $e^* \in E_\cH^1(A',B') \cup E_\cH^{k-1}(A',B')$. Since $|e^* \cap A'| \in \{1 , k-1 \}$ and $k$ is even, we have $|A'| \equiv 1 \equiv |e^* \cap A'|\:({\rm mod}\:2)$.
    Thus, for both cases, $|A''| = |A'| - |e^* \cap A'| \equiv 0\:({\rm mod}\:2)$, so $(A'',B'')$ satisfies the divisibility condition with respect to the type $\rm (c)$.
{\bf Case $\alpha = {\rm (d)}$}.
    In this case, $k \equiv 0 \modu{4}$ and $r^* = k/2 + 1$ is odd by Definitions~\ref{def:type} and~\ref{def:special}.
    If $\frac{|A'| - |B'|}{2} \equiv \frac{n}{k}\:({\rm mod}\:2)$ then by~\ref{S2}, there are at least $\eps n^{k-1}$ edges in $E_\cH^{r^*}(A',B')$ incident to a vertex $v \in V(\cH)$. We choose $e^* \in E_\cH^{r^*}(A',B')$. Then $\frac{|A''|-|B''|}{2} = \frac{|A'|-|B'|}{2} - 1 \equiv \frac{n}{k} - 1 \equiv \frac{n-k}{k}\:({\rm mod}\:2)$, so $(A'',B'')$ satisfies the divisibility condition with respect to the type $\rm (d)$.
    Otherwise if $\frac{|A'| - |B'|}{2} \not\equiv \frac{n}{k}\:({\rm mod}\:2)$, by Lemma~\ref{lem:manyatypical}, $|E_\cH^2(A',B') \cup E_\cH^{k-2}(A',B')| \geq \eps n^{k-1}$. Thus, we choose $e^* \in E_\cH^2(A',B') \cup E_\cH^{k-2}(A',B')$.
    Since $4 \mid k$, we have $\frac{|A''|-|B''|}{2} \equiv \frac{|A'|-|B'| + k - 4}{2} \equiv \frac{|A'|-|B'|}{2} \equiv \frac{n}{k} - 1 =  \frac{n-k}{k}\:({\rm mod}\:2)$, so $(A'',B'')$ satisfies the divisibility condition with respect to the type $\rm (d)$.
{\bf Case $\alpha = {\rm (e)}$}. 
    In this case, $k \equiv 2 \modu{4}$ and $r^* = k/2$ is odd by Definitions~\ref{def:type} and~\ref{def:special}.
    If $\frac{|A'| - |B'|}{2}$ is even then by~\ref{S2}, there are at least $\eps n^{k-1}$ edges in $E_\cH^{r^*}(A',B')$ incident to a vertex $v \in V(\cH)$. We choose $e^* \in E_\cH^{r^*}(A',B')$. Then $\frac{|A''|-|B''|}{2} = \frac{|A'|-|B'|}{2} \equiv 0\:({\rm mod}\:2)$, so $(A'',B'')$ satisfies the divisibility condition with respect to the type $\rm (e)$.
    Otherwise if $\frac{|A'| - |B'|}{2}$ is odd, by Lemma~\ref{lem:manyatypical}, $|E_\cH^2(A',B') \cup E_\cH^{k-2}(A',B')| \geq \eps n^{k-1}$. We choose $e^* \in E_\cH^2(A',B') \cup E_{\cH}^{k-2}(A',B')$. 
    Since $k \equiv 2 \modu{4}$, we have $\frac{|A''|-|B''|}{2} \equiv \frac{|A'|-|B'| + k - 4}{2} \equiv \frac{|A'|-|B'|}{2} + 1 \equiv 0\:({\rm mod}\:2)$, so $(A'',B'')$ satisfies the divisibility condition with respect to the type $\rm (e)$.}
\end{claimproof}

Now we fix $e^* \in \cH$ satisfying Claim~\ref{claim:div''}.
Let us define
\begin{align}\label{def:h'}
    \cH' \coloneqq 
    \begin{cases}
        \cH \cap \cH^0(k,A'',B'') & \text{if } \alpha \in \{ {\rm (a)}, {\rm (d)}, {\rm (e)} \},\\
        \cH \cap \overline{\cH^0(k,A'',B'')} & \text{if } \alpha \in \{ {\rm (b)}, {\rm (c)} \}.
    \end{cases}
\end{align}

Thus, the subhypergraph $\cH'$ is the collection of the $\alpha$-typical edges in $\cH - V(e^*)$ with respect to $(A'',B'')$. Since $\cH - V(e^*)$ belongs to the type $\alpha$ with respect to $(6k\eps^{1/2},A'',B'')$ by Claim~\ref{claim:div''} and Observation~\ref{obs:epscontain_type},\COMMENT{we have
\begin{itemize}
    \item $|\cH^0(k,A'',B'') \setminus \cH'| \leq 6k\eps^{1/2}n^k$ if  $\alpha \in \{ {\rm (a)}, {\rm (d)}, {\rm (e)} \}$ and
    \item $|\overline{\cH^0(k,A'',B'')} \setminus \cH'| \leq 6k\eps^{1/2}n^k$ if $\alpha \in \{ {\rm (b)}, {\rm (c)} \}$,
\end{itemize}}
$\cH'$ also belongs to the type $\alpha$ with respect to $(6k\eps^{1/2} , A'' , B'')$.

\begin{claim}\label{claim:proph'}
The hypergraph $\cH'$ satisfies the following properties.
\begin{itemize}
    \item At least a $(1-\eps^{1/6})$-fraction of $(k-1)$-sets $S \in \binom{A'' \cup B''}{k-1}$ satisfy $d_{\cH'}(S) \geq n/2 - 10\eps^{1/6} kn$.
    
    \item For each vertex $v \in A'' \cup B''$, $d_{E_{\cH'}^{r^*}(A'',B'')}(v) > 0.15 d_{\cK_{r^*}(A'',B'')}(v) \geq 0.15 \frac{n^{k-1}}{3^{k-1} (k-1)!}$.\COMMENT{Note that $d_{\cK_{r^*}(A'',B'')}(v) \geq \frac{\min(|A''|-k+1 , |B''|-k+1)^{k-1}}{(k-1)!} \geq \frac{(n/3)^{k-1}}{(k-1)!}.$}
\end{itemize}
In particular, since $\eps \ll \eta$, $\cH'$ is $(1/2 - \eta ,\: \frac{0.15}{3^{k-1}(k-1)!}, \: k - 1 ,\: \eta)$-dense.
\end{claim}
\begin{claimproof}
Without loss of generality, we may assume that $\alpha \in \{ {\rm (a)}, {\rm (d)}, {\rm (e)} \}$. For the other case $\alpha \in \{ {\rm (b)}, {\rm (c)} \}$, we can just switch the role of $\cH^0(k,A'',B'')$ and $\overline{\cH^0(k,A'',B'')}$.

For any $k-1$ distinct vertices $v_1 , \dots , v_{k-1} \in A'' \cup B''$, depending on the parity of $|A'' \cap \{v_1 , \dots , v_{k-1} \}|$, $d_{\cH^0(k,A'',B'')}(v_1 , \dots , v_{k-1})$ is either $|A'' \setminus \{v_1 , \dots , v_{k-1} \}|$ or $|B'' \setminus \{v_1 , \dots , v_{k-1} \}|$. Thus, since by \ref{S1}, $\min \{ |A''|,|B''| \} \geq \min \{ |A'|,|B'| \} - k \geq n/2 - 2\eps^{1/2}kn - k$ and $\max \{ |A''|,|B''| \} \leq \max \{ |A'|,|B'| \} \leq n/2 + 2 \eps^{1/2} kn$, we have
\begin{itemize}
    \item $\delta_{k-1}(\cH^0(k,A'',B'')) \geq \min \{ |A''|,|B''| \} - (k-1) \geq n/2 - 3 \eps^{1/2}kn$, and 
    
    \item $\Delta_{k-1}(\cH^0(k,A'',B'')) \leq \max \{ |A''|,|B''| \} \leq n/2 + 2 \eps^{1/2}kn$,    
\end{itemize}
where $\Delta_{k-1}(\cH^0(k,A'',B'')) \coloneqq \max \{ d_{\cH^0(k,A'',B'')}(S) : S \in \binom{A'' \cup B''}{k-1} \}$ is the maximum codegree of $\cH^0(k,A'',B'')$.

Since $\cH' \subseteq \cH^0(k,A'',B'')$, every $(k-1)$-set $S \in \binom{A'' \cup B''}{k-1}$ satisfies $d_{\cH'}(S) \leq \Delta_{k-1}(\cH^0(k,A'',B'')) \leq n/2 + 2 \eps^{1/2} kn$.
Let $N$ be the number of $(k-1)$-sets $S \in \binom{A'' \cup B''}{k-1}$ such that $d_{\cH'}(S) \geq n/2 - 10 \eps^{1/6} kn$.
Since $|\cH^0(k,A'',B'') \setminus \cH'| \leq 6k\eps^{1/2}n^k$,
\begin{align*}
    ke(\cH') \geq ke(\cH^0(k,A'',B'')) - 6k^2 \eps^{1/2} n^k &\geq \binom{|A'' \cup B''|}{k-1} \delta_{k-1}(\cH^0(k,A'',B'')) - 6k^2 \eps^{1/2} n^k\\
    &\geq \binom{|A'' \cup B''|}{k-1} \left(n/2 - 3\eps^{1/2} kn - 2(k - 1)!6k^2 \eps^{1/2}n\right).
\end{align*}
On the other hand,
\begin{align*}
    ke(\cH') = \sum_{S \in \binom{A'' \cup B''}{k-1}} d_{\cH'}(S) &\leq \left ( \binom{|A'' \cup B''|}{k-1} - N \right) \left(n/2 - 10\eps^{1/6} kn\right) + N \left(n/2 + 2\eps^{1/2} kn \right) \\
    & = \binom{|A'' \cup B''|}{k-1} \left(n/2 - 10\eps^{1/6} kn\right) + N \left(10 \eps^{1/6}k + 2 \eps^{1/2} k\right)n.
\end{align*}
Combining both inequalities, since $\eps \ll 1/k$,
\begin{align*}
    N &\geq \binom{|A'' \cup B''|}{k-1} \frac{10 \eps^{1/6} k - 3\eps^{1/2} k - 2(k-1)! 6k^2 \eps^{1/2} }{10 \eps^{1/6}k + 2 \eps^{1/2} k} \geq \binom{|A'' \cup B''|}{k-1} \frac{10\eps^{1/6}k - \eps^{1/3} k}{10\eps^{1/6}k}\\ 
    & > (1 - \eps^{1/6}) \binom{|A'' \cup B''|}{k-1},
\end{align*}
as desired.

Note that $r^*$ is the special $\alpha$-typical index for the hypergraphs $\cH$, $\cH - V(e^*)$, and $\cH'$. 
Since $\cH'$ is the subhypergraph of typical edges of $\cH - V(e^*)$, we have $E_{\cH'}^{r^*}(A'',B'') = E_{\cH - V(e^*)}^{r^*}(A'',B'')$. Moreover, since $|A' \setminus A''| + |B' \setminus B''| = |V(e^*)| = k$, we have $d_{E_{\cH'}^{r^*}(A'',B'')}(v) \geq d_{E_{\cH}^{r^*}(A',B')}(v) - kn^{k-2}$ for each vertex $v \in A'' \cup B''$. Thus, by~\ref{S1}, we have $d_{E_{\cH'}^{r^*}(A'',B'')}(v) > 0.15 d_{\cK_{r^*}(A'',B'')}(v)$ as desired.
\end{claimproof}

\begin{claim}\label{claim:divfinal}
Let $M'$ be a matching in $\cH'$ such that $|M'| \equiv 0\:({\rm mod}\:2)$ if $\alpha \in \{ {\rm (d)} , {\rm (e)} \}$. Then the ordered partition $(A'' \setminus V(M') , B'' \setminus V(M'))$ satisfies the divisibility condition with respect to $\alpha$.
\end{claim}
\begin{claimproof}
By Claim~\ref{claim:div''}, $(A'',B'')$ satisfies the divisibility condition with respect to the type $\alpha$. Now we divide the cases according to the type $\alpha$.

{\bf Case $\alpha \in \{ {\rm (a)} , {\rm (c)} \}$}. Since $|e \cap A''| \equiv 0\:({\rm mod}\:2)$ for each $e \in \cH'$, we have $|A''| \equiv |A'' \setminus V(M')|\:({\rm mod}\:2)$. 
    
{\bf Case $\alpha = {\rm (b)}$}. Since $|e \cap B''| \equiv 0\:({\rm mod}\:2)$ for each $e \in \cH'$, we have $|B''| \equiv |B'' \setminus V(M')|\:({\rm mod}\:2)$. 
    
{\bf Case $\alpha \in \{ {\rm (d)} , {\rm (e)} \}$}. Let $M' = \{e_1 , \dots , e_t \}$ for some even integer $t$. Let $\ell_i \coloneqq |e_i \cap A''|$ for each $i \in [t]$. Since $|e \cap A''|$ is odd for each $e \in \cH'$, we have $\ell_i \equiv 1\:({\rm mod}\:2)$ for each $i \in [t]$. Thus,
\begin{align*}
    |A'' \setminus V(M')| = |A''| - (\ell_1 + \dots + \ell_t)\text{ and }
    |B'' \setminus V(M')| = |B''| - kt + (\ell_1 + \dots + \ell_t),
\end{align*}
so $\frac{|A'' \setminus V(M')| - |B'' \setminus V(M')|}{2} = \frac{|A''|-|B''|}{2} + k\frac{t}{2} - (\ell_1 + \dots + \ell_t) \equiv \frac{|A''|-|B''|}{2}\:({\rm mod}\:2)$. 
Thus, $(A'' \setminus V(M') , B'' \setminus V(M'))$ satisfies the divisibility condition with respect to $\alpha$.
\end{claimproof}

Now we have all the ingredients to prove Lemma~\ref{lem:extremal_OM}.
By Claim~\ref{claim:proph'},~\ref{O1} holds. To show~\ref{O2}, it suffices to prove the following claim.
Recall that $\ell \coloneqq \lceil \frac{k-1}{k} \log_2 n \rceil$, $C_\ell \coloneqq \sum_{i=1}^{\ell}2^{-i} = 1 - 2^{-\ell}$, and $p_\ell \coloneqq 1/(C_\ell 2^{\ell})$.

\begin{claim}
Let $U_\ell$ be a $p_\ell$-random subset of $V(\cH') = V(\cH) \setminus V(e^*)$. With probability $1-o(1)$, for all matchings $M'$ of $\cH'$ satisfying $2 \mid |M'|$, $V(\cH') \setminus V(M') \subseteq U_\ell$, and $|U_\ell \cap V(M')| \leq \eps |U_\ell|$, the subhypergraph $\cH'' \coloneqq \cH - V(e^*) - V(M')$ has a perfect matching.
\end{claim}
\begin{claimproof}
First of all, by a Chernoff bound (Lemma~\ref{lem:chernoff}), $|U_\ell| = (1 \pm \eps) p_\ell n$ with probability $1-o(1)$.
We apply Theorem~\ref{thm:perfectmatching_extremal} to show that $\cH''$ has a perfect matching. To do so, we will show that the following assumptions of Theorem~\ref{thm:perfectmatching_extremal} hold with probability $1-o(1)$, where $A''' \coloneqq A'' \setminus V(M')$ and $B''' \coloneqq B'' \setminus V(M')$.

\begin{enumerate}
    \item\label{cond1} $\delta_{k-1}(\cH'') \geq |U_\ell|/2 - \eta |U_\ell|$.
    
    \item\label{cond2} $\cH''$ belongs to the type $\alpha$ with respect to $(\eta , A''' , B''')$. Thus, in particular, $r^*$ is the special $\alpha$-typical index for $\cH''$.
    
    \item\label{cond3} $(A''' , B''')$ satisfies the divisibility condition with respect to $\alpha$.
    
    \item\label{cond4} For each vertex $v \in V(\cH'')$, $d_{E_{\cH''}^{r^*}(A''', B''')}(v) > 0.1 d_{\cK_{r^*}(A''',B''')}(v)$.
\end{enumerate}

First of all, Claim~\ref{claim:divfinal} shows~\eqref{cond3}.
Now we prove~\eqref{cond2}.
Since $\cH'$ belongs to the type $\alpha$ with respect to $(6k\eps^{1/2} , A'' , B'')$ (see the discussion below~\eqref{def:h'}), let us define $\cF \subseteq \binom{A'' \cup B''}{k}$ such that
\begin{itemize}
    \item $\bigcup_{\text{$r$: typical}}\cK_r(A'',B'') \setminus \cH' \subseteq \cF$ and
    \item $|\cF| = 6k \eps^{1/2} n^k \pm 1$.
\end{itemize}

In particular, $\cF$ contains all possible typical `non-edges' of $\cH'$.
By Lemma~\ref{lemma:prob_lemma} \ref{typical_i}, $U_\ell$ is $(p_\ell , \eps , \cF)$-typical with probability $1-o(1)$, so the number of elements in $\cF$ contained in $U_\ell$ is $(1 \pm \eps)p_\ell^k |\cF| \leq 7k \eps^{1/2} |U_\ell|^k$ with probability $1-o(1)$.
Note that the number of elements in $\cF$ contained in $U_\ell$ is at least the number of all possible typical `non-edges' of $\cH'$ contained in $U_\ell$. Thus, $|\bigcup_{\text{$r$: typical}}\cK_r(A'' \cap U_\ell , B'' \cap U_\ell) \setminus \cH'[U_\ell]| \leq 7k \eps^{1/2} |U_\ell|^k$, so $\cH'[U_\ell]$ belongs to the type $\alpha$ with respect to $(7k \eps^{1/2} , A'' \cap U_\ell,B'' \cap U_\ell)$.
Since $\cH' \subseteq \cH$ and $\eps \ll \eta \ll 1/k$, $\cH[U_\ell]$ belongs to the type $\alpha$ with respect to $(\eta / 2 , A'' \cap U_\ell, B'' \cap U_\ell)$. 
Thus, since $\eps \ll \eta$ and $|U_\ell \cap V(M')| \leq \eps |U_\ell|$ and $(V(\cH) \setminus V(e^*)) \setminus V(M') \subseteq U_\ell$, $\cH[U_\ell \setminus V(M')] = \cH''$ belongs to the type $\alpha$ with respect to $(\eta , A'' \setminus V(M'), B'' \setminus V(M'))$, proving~\eqref{cond2}.

Now we prove~\eqref{cond1}. 
Since $\mathbb{E}[d_\cH(S;U_\ell)] \geq p_\ell (\delta_{k-1}(\cH) - |V(e^*)|)$ for each $S \in \binom{V(\cH')}{k-1}$ and $|U_\ell| = (1 \pm \eps)p_\ell n$ with probability $1-o(1)$, by a Chernoff bound (Lemma~\ref{lem:chernoff}) and a union bound, we have $\delta_{k-1}(\cH[U_\ell]) \geq (1 - \eta)|U_\ell| / 2$ with probability $1-o(1)$. Thus, $\delta_{k-1}(\cH[U_\ell] - V(M')) \geq (1 - \eta)|U_\ell| / 2 - |U_\ell \cap V(M')| > |U_\ell|/2 - \eta |U_\ell|$ with probability $1-o(1)$, which shows~\eqref{cond1}.

Finally, we prove~\eqref{cond4}.
For each $v \in A'' \cup B''$, since $U_\ell$ is a $p_\ell$-random subset of $V(\cH')$, we have
\begin{itemize}
    \item $\mathbb{E}[d_{E_{\cH'[U_\ell]}^{r^*}(A'' \cap U_\ell , B'' \cap U_\ell)}(v)] = p_\ell^{k-1} d_{E_{\cH'}^{r^*}(A'', B'')}(v)$ and
    \item $\mathbb{E}[d_{\cK_{r^*}(A'' \cap U_\ell , B'' \cap U_\ell)}(v)] = p_\ell^{k-1} d_{\cK_{r^*}(A'', B'')}(v)$.
\end{itemize}

Let $\cF_v \coloneqq \{ e \setminus \{ v \} : v \in e \in E_{\cH'}^{r^*}(A'' , B'') \}$, and let $\cG_v \coloneqq \{ e \setminus \{ v \} : v \in e \in \cK_{r^*}(A'' , B'') \}$. Applying Lemma~\ref{lemma:prob_lemma} \ref{typical_i} twice for each $v \in A'' \cup B''$ and taking union bounds, with probability $1 - o(1)$, $U_\ell$ is both $(p_\ell ,\eps,\cF_v)$-typical and $(p_\ell ,\eps,\cG_v)$-typical for all $v \in A'' \cup B''$. Thus, for each $v \in U_\ell$,
\begin{align}
    d_{E_{\cH'[U_\ell]}^{r^*}(A'' \cap U_\ell , B'' \cap U_\ell)}(v) &= (1 \pm \eps) p_\ell^{k-1} d_{E_{\cH'}^{r^*}(A'', B'')}(v) \nonumber \\
    \overset{\text{\cref{claim:proph'}}} & {\geq} (1 - \eps) p_\ell^{k-1} \cdot 0.15 d_{\cK_{r^*}(A'', B'')}(v) \nonumber \\
    & \geq 0.15 \cdot \frac{1-\eps}{1+\eps} \cdot d_{\cK_{r^*}(A'' \cap U_\ell, B'' \cap U_\ell)}(v), \label{eqn:degree_in_randomset}
\end{align}
where the first equality and the last inequality follow since $U_\ell$ is $(p_\ell,\eps,\cF_v)$-typical and $(p_\ell,\eps,\cG_v)$-typical, respectively. On the other hand, since $r^*$ is the special $\alpha$-typical index for $\cH'$ and $\cH'$ is the subhypergraph of typical edges of $\cH - V(e^*)$, we have $E_{\cH''}^{r^*}(A''' , B''') = E_{\cH' - V(M')}^{r^*}(A''' , B''')$. Thus,
\begin{align*}
    d_{E_{\cH''}^{r^*}(A''' , B''')}(v) & \geq d_{E_{\cH'[U_\ell]}^{r^*}(A'' \cap U_\ell , B'' \cap U_\ell)}(v) - |U_\ell \cap V(M')| |U_\ell|^{k-2}\\
    \overset{\eqref{eqn:degree_in_randomset}}&{\geq} 0.12 d_{\cK_{r^*}(A'' \cap U_\ell , B'' \cap U_\ell)}(v) - \eps |U_\ell|^{k-1} \\
    & > 0.1 d_{\cK_{r^*}(A'' \cap U_\ell , B'' \cap U_\ell)}(v) \geq 0.1 d_{\cK_{r^*}(A''' , B''')}(v).
\end{align*}

In the penultimate inequality we used that $|U_\ell|$ is large enough (it is $(1 \pm \eps)p_\ell n$ with probability $1-o(1)$), and in the final inequality we used $V(\cH') \setminus V(M') \subseteq U_\ell$.
This proves~\eqref{cond4}. Thus, by Theorem~\ref{thm:perfectmatching_extremal}, $\cH''$ has a perfect matching with probability $1-o(1)$.
\end{claimproof}
\end{proof}

\bibliographystyle{amsabbrv}
\bibliography{ref}

\APPENDIX{

\appendix
\renewcommand\sectionname{}
\renewcommand\appendixname{A}
\renewcommand{\thesection}{\Alph{section}}

\section{Proofs of Lemmas~\ref{lemma:prob_lemma}, \ref{lem:reduced_degree}, and \ref{lem:vortex_existence}}\label{app:lemma-proofs}

In this section, we prove Lemmas~\ref{lemma:prob_lemma}, \ref{lem:reduced_degree}, and \ref{lem:vortex_existence}.

As mentioned, we prove \cref{lemma:prob_lemma} via the polynomial concentration theorem of Kim and Vu \cite{Kim2000}. We first give some definitions and then state the theorem. Let $n$ and $r$ be integers and let $\cG$ be a hypergraph on $n$ vertices in which each edge has size at most $r$. Suppose $\{ X_v  : v \in V(\cG) \}$ is a set of mutually independent Bernoulli random variables. We define the random variable
\[
Y_\cG \coloneqq \sum_{e \in \cG} \prod_{v \in e} X_v.
\]
For a subset $A \subseteq V(\cG)$, we define $\cG_A$ to be the hypergraph with $V(\cG_A) \coloneqq V(\cG) \setminus A$ and $E(\cG_A) \coloneqq \{S \subseteq V(\cG_A) \colon S \cup A \in E(\cG)\}$. Thus we have
\[
Y_{\cG_A} = \sum_{\substack{e \in \cG \\ A \subseteq e}} \prod_{v \in e \setminus A} X_v.
\]
Moreover, for each $0 \leq i \leq r$, we let
\[
\cE_i(\cG) \coloneqq \max_{\substack{A \subseteq V(\cG) \\ \s{A} = i}} \mathbb{E}[Y_{\cG_A}].
\]
Finally, we let $\cE(\cG) \coloneqq \max_{0 \leq i \leq r} \cE_i(\cG)$ and $\cE'(\cG) \coloneqq \max_{1 \leq i \leq r} \cE_i(\cG)$.

\begin{theorem}[Kim--Vu polynomial concentration \cite{Kim2000}] \label{thm:Kim_Vu}
In the above setting, we have
\[
\pr{\s{Y_\cG - \mathbb{E}[Y_\cG]} > a_r (\cE(\cG)\cE'(\cG))^{1/2}\lambda^r} \leq 2e^2 e^{-\lambda} n^{r-1}
\]
for any $\lambda > 1$ and $a_r \coloneqq 8^r r!^{1/2}$.
\end{theorem}

\begin{proof}[Proof of \cref{lemma:prob_lemma}]
    We first prove \labelcref{typical_i}.
    Independently for each $v \in V$, let $X_v \in \{0,1\}$ with $\pr{X_v = 1} = p$ and let $U = \{v \in V \colon X_v = 1\}$. Define~$\cG$ to be the hypergraph with $V(\cG) = V$ and $E(\cG) = \mathcal{F}$. Note that each edge in $\cG$ has size $s$.
    Since~$Y_\cG$ is the number of elements of $\mathcal{F}$ that are contained in $U$, we have
    \begin{align*}
        \mathbb{E}[Y_\cG] = \s{\cF} p^s \geq \eps (np)^{s-1/2}.
    \end{align*}
    Let $1 \leq i \leq s$ and $A \subseteq V(\cG) = V$ with $\s{A} = i$. Note that $Y_{\cG_A}$ is the number of $F \in \cF$ such that $A \subseteq F$ and $F \setminus A \subseteq U$. It follows that 
    \begin{align*}
        \mathbb{E}[Y_{\cG_A}] \leq n^{s-i} p^{s-i} = (np)^{s-i} \leq (np)^{s-1} \leq (np)^{-1/4}\mathbb{E}[Y_\cG].
    \end{align*}    
    Hence $\cE(\cG) = \mathbb{E}[Y_\cG]$ and $\cE'(\cG) \leq (np)^{-1/4}\mathbb{E}[Y_\cG]$. Now let 
    \begin{align*}
        \lambda \coloneqq \left(\frac{\eps\mathbb{E}[Y_\cG]}{ a_{s}(\cE(\cG)\cE'(\cG))^{1/2}}\right)^{1/s} \geq \left(\frac{\eps(np)^{1/8}}{a_{s}}\right)^{1/s} \geq n^{\beta/(9s)}.
    \end{align*}
    By \cref{thm:Kim_Vu}, we have
    \begin{align*}
        \pr{\s{Y_\cG - \mathbb{E}[Y_\cG]} > \eps \mathbb{E}[Y_\cG]} &= \pr{\s{Y_\cG - \mathbb{E}[Y_\cG]} > a_{s} (\cE(\cG)\cE'(\cG))^{1/2} \lambda^{s}} \\
        &\leq 2e^2 e^{-\lambda} n^{s-1} \leq \exp(-n^{\beta/(10s)}).
    \end{align*}
    Thus with probability at least $1 -  \exp(-n^{\beta/(10s)})$, we have that the number of elements of $\cF$ contained in~$U$ is 
    $(1 \pm \eps) \mathbb{E}[Y_\cG] = (1 \pm \eps)\s{\cF}p^s$, 
    which concludes the proof.

    Now we show that \labelcref{typical_ii} follows from \labelcref{typical_i}. Let $\mathcal{F}' \subseteq \binom{V}{s}$ be such that $\mathcal{F} \subseteq \mathcal{F}'$ and $\eps n^s (np)^{1/2} \leq \s{\mathcal{F}'} \leq \eps n^s$. By \labelcref{typical_i}, with probability at least $1 -  \exp(-n^{\beta/(10s)})$, $U$ is $(p, \eps, \mathcal{F}')$-typical. It follows that, with probability at least $1 -  \exp(-n^{\beta/(10s)})$,
    \[
        \s{\{S \in \mathcal{F} \colon S \subseteq U\}} \leq \s{\{S \in \mathcal{F}' \colon S \subseteq U\}} \leq (1+\eps) p^s \s{\mathcal{F}'} \leq 2 \eps (np)^s,
    \]
    as desired.
\end{proof}


\begin{proof}[Proof of \cref{lem:reduced_degree}]
Let $S = \{i_1 , \dots , i_{d} \} \in \binom{[t]}{d}$ be \emph{good} if there are at least $(1 - \eps^{1/2}) \binom{t-d}{k-d}$ many  $(k-d)$-sets $\{i_{d+1} , \dots , i_k \} \in \binom{[t] \setminus S}{k-d}$ such that $(V_{i_1} , \dots , V_{i_k})$ is $\eps$-regular. 
Since there are at most $\eps \binom{t}{k}$ many $k$-sets in $\binom{[t]}{k}$ which are not $\eps$-regular, by an averaging argument, all but at most $\frac{\eps \binom{t}{k} \binom{k}{d}}{\eps^{1/2} \binom{t-d}{k-d}} = \eps^{1/2} \binom{t}{d}$ many $d$-sets in $\binom{[t]}{d}$ are good.

Now it suffices to show that every good set in $\binom{[t]}{d}$ has $d$-degree at least $(c - \gamma)\binom{t-d}{k-d}$ in $\cR$.
Suppose, for a contradiction, that a good set $S = \{ i_1 , \dots , i_{d} \} \in \binom{[t]}{d}$ has $d$-degree less than $(c - \gamma) \binom{t-d}{k-d}$ in~$\cR$.
Let $n_* \coloneqq |V_1| = \dots = |V_t|$, which satisfies $\frac{2n}{3t} \leq (1-\eps)n/t \leq n_* \leq n/t$ since $\eps \leq 1/3$. 
Let $N_S$ be the set of edges $e \in \cH$ with $|e \cap V_{i_j}| = 1$ for all $j \in [d]$. 
Since all but at most $\eta n^{d}$ many $d$-sets in $\binom{V_{i_1} \cup \cdots \cup V_{i_d}}{d}$ have $d$-degree at least $c \binom{n-d}{k-d}$, we have
\begin{align*}
    |N_S| &\geq (n_*^{d} - \eta n^{d}) c \binom{n-d}{k-d} - n_*^{d} \cdot d n_* \binom{n-d-1}{k-d-1} \geq n_*^{d} \binom{n-d}{k-d} \left (c - \gamma c/6 - dkn_* / n \right )\\
    &\geq n_*^{d} (c - \gamma/3) \binom{n-d}{k-d}.
\end{align*}

Let $\mathcal{E}(S)$ be the set of $(k-d)$-sets $\{i_{d+1} , \dots , i_k \} \in \binom{[t] \setminus S}{k-d}$ such that $(V_{i_1} , \dots , V_{i_k})$ is not $\eps$-regular.
Since $S$ is good, we have $|\mathcal{E}(S)| \leq \eps^{1/2}\binom{t-d}{k-d}$. Since $\cR$ is the $(\gamma/3 , \eps)$-reduced hypergraph, for $\{i_{d+1}, \dots, i_k\} \in \binom{[t]\setminus S}{k-d} \setminus (N_\cR(S) \cup \mathcal{E}(S))$, we have $e_\cH(V_{i_1} , \dots , V_{i_k}) \leq \gamma/3 \cdot |V_{i_1}| \cdots |V_{i_k}| = \gamma/3 \cdot n_*^k$. 
Note that moreover there are at most $\eps n_*^d \cdot n^{k-d}$ edges $e \in N_S$ with $e \cap V_0 \neq \varnothing$. Finally, there are at most $t^{k-d-1} n_*^k$ edges $e \in N_S$ with $e \cap V_0 = \varnothing$ that contain more than one vertex from $V_i$ for some $i \in [t]$. Recall that by assumption $|N_\cR(S)| < (c-\gamma) \binom{t-d}{k-d}$.
Hence we have
\begin{align*}
    |N_S| &\leq \s{\binom{[t]\setminus S}{k-d} \setminus (N_\cR(S) \cup \mathcal{E}(S))} \gamma/3 \cdot n_*^k + \s{N_\cR(S) \cup \mathcal{E}(S)} n_*^k + \eps n_*^d \cdot n^{k-d} + t^{k-d-1}n_*^k\\
    &< \binom{t-d}{k-d} \gamma n_*^k/3 + (c- \gamma + \eps^{1/2}) \binom{t-d}{k-d} n_*^k + \eps n_*^d \cdot n^{k-d} + t^{k-d-1}n_*^k \\
    & < n_*^{d} (c - \gamma / 3) \binom{n-d}{k-d}.
\end{align*}
This contradicts the bound $N_S \geq n_*^{d} (c - \gamma/3) \binom{n-d}{k-d}$ obtained above. Thus every good set in $\binom{[t]}{d}$ has $d$-degree at least $(c - \gamma)\binom{t-d}{k-d}$ in $\cR$.
\end{proof}

\begin{proof}[Proof of \cref{lem:vortex_existence}]
Note that $p_i n \geq p_\ell n \geq \eps n^{1/k}$ for all $i \in [\ell]$.
For each $i \in [\ell]$, since $\mathbb{E}[|U_i|] = p_i n$, by a Chernoff bound and a union bound, with probability at least $1-\exp(-n^{1/(2k)})$, for all $i \in [\ell]$ we have $|U_i| = (1 \pm \eps)p_i n$. Thus \aas~\ref{V1} holds.

We call $S \in \binom{V(\cH)}{d}$ \emph{good} if $d_\cH(S) \geq \alpha_1 n^d$, otherwise we call it \emph{bad}.
Since $\cH$ is $(\alpha_1 , \alpha_2 , d , \eps)$-dense, there are at most $\eps n^{d}$ bad $d$-sets in $\binom{V(\cH)}{d}$. 
By \cref{lemma:prob_lemma} \ref{typical_ii} and a union bound, we have that, with probability at least $1 - \exp(-n^{1/(11k^2)})$, for each $i \in [\ell]$, $U_i$ contains at most $2\eps (p_in)^d$ bad $d$-sets. By \cref{lemma:prob_lemma} \ref{typical_i} and a union bound, we have that, with probability at least $1 - \exp(-n^{1/(11k^2)})$, for each $i \in [\ell]$ and each good $S \in \binom{V(\cH)}{d}$, we have $d_\cH(S; \binom{U_i}{k-d}) \geq (\alpha_1 - 2\eps) (p_in)^{k-d}$. Hence \aas~\ref{V2} holds.

By \cref{lemma:prob_lemma} \ref{typical_i} and a union bound, we have that, with probability at least $1 - \exp(-n^{1/(11k^2)})$, for each $i \in [\ell]$ and each vertex $v \in V(\cH)$, $d_\cH(v ; \binom{U_{i} \setminus \{ v \} }{k-1}) \geq (\alpha_2 - 2\eps) (p_{i} n)^{k-1}$. So \aas~\ref{V3} holds. 
\end{proof}

}

\end{document}